\documentclass[a4paper,12pt]{article}

\newtheorem{theorem}{Theorem}[section] 
\newtheorem{remark}[theorem]{Remark} 
\newtheorem{lemma}[theorem]{Lemma} 
\newtheorem{definition}[theorem]{Definition} 
\newtheorem{proposition}[theorem]{Proposition} 
 
\newenvironment{proof}[1][Proof]{\textbf{#1.} }{\  
 \rule{0.5em}{0.5em}\\}

\newcommand{\R}{\mathbb{R}}
\newcommand{\N}{\mathbb{N}}

\newcommand{\intbar}{-\kern-1em\int}

\usepackage{amsfonts}
\usepackage{xcolor}

\usepackage{hyperref}

\title{Nonlinear balance and asymptotic behavior of supercritical 
  reaction-diffusion equations with nonlinear boundary conditions
}

\author{An\'{\i}bal Rodr\'{\i}guez-Bernal
\thanks{Partially supported by Project MTM2009--07540, MEC and   GR58/08
  Grupo 920894, UCM, Spain.}
\and
Alejandro Vidal-L\'opez
\thanks{Supported by Marie Curie IEF 252078 ULD3DNSE, EU}}

\newcommand{\eps}{\varepsilon}
\renewcommand{\epsilon}{\varepsilon}

\begin{document}

\maketitle

\begin{center}
${}^{*}$Departamento de Matem\'atica Aplicada \\
Universidad Complutense de Madrid, \\ Madrid 28040,  SPAIN \\ and \\
Instituto de Ciencias Matem\'aticas \\
CSIC-UAM-UC3M-UCM \\ 
arober@mat.ucm.es \\ \mbox{}\\
${}^{\dag}$Mathematics Institute \\ 
University of Warwick. UK. \\
A.Vidal-Lopez@warwick.ac.uk
\end{center}

\section{Introduction} 

When one considers reaction diffusion problems with nonlinear boundary (flux)
conditions one typically faces problems in which, in a natural way, two
different nonlinear mechanisms, of very different nature,
compete. Namely, 
interior reaction and boundary flux.  In this context it is therefore a natural
question to understand which is the nonlinear balance between these two
competing nonlinear mechanisms. From the mathematical point of view a delicate
technical problem is also to determine a large class of initial data for which
the problem can be solved. Once this is done the balance between the nonlinear
terms will determine the subsequent behavior of the solutions.

\medskip

Let us consider the following problem
  \begin{equation}
  \label{eq:pbFG:monotonic}
  \left\{\begin{array}{rclcl}
    u_t  - \Delta u + f(u) & =
    & 0 & \mathrm{in} & \Omega\\
    \displaystyle\frac{\partial u}{\partial \vec{n}}&=& g(u) &
    \mathrm{on} & \Gamma\\
    u(0) &=& u_0
  \end{array}\right.
\end{equation}
in a bounded domain $\Omega \subset \R^N$, $N\geq 1$.  The
prototype nonlinearities we consider behave like
\begin{displaymath}
\label{eq:model:nonlinearities}
  f(s) \sim c_{f} |s|^{p-1} s,\qquad g(s) \sim c_{g} |s|^{q-1} s
\end{displaymath}
for $|s| \to \infty$, with $c_{f}, c_{g} >0$. More precisely, we
suppose  that $f$ and $g$ satisfy
\begin{equation}
  \label{eq:mainHyp:fg}
  pc_f |s|^{p-1} - A_0\leq f^\prime(s) \leq pC_f|s|^{p-1} + A_1 , 
\qquad 
  qc_g |s|^{q-1} - B_0\leq g^\prime(s) \leq qC_g|s|^{q-1} + B_1,
\end{equation}
for some $c_f, c_g ,C_f,C_g>0$, $A_0, A_1,B_0,B_1>0$, $1<p,q<\infty$.

These assumptions imply that there is an actual competition in
(\ref{eq:pbFG:monotonic}) between nonlinear terms. On one hand, $f$
has a dissipative character and tries to make solutions global and
bounded. On the other hand $g$ has an explosive nature and tries to
make solutions blow up in finite time; see \cite{chipotFilaQuittner91:_station,Tajdine2001}. As it is
currently known, and explained below, the nonlinear balance between these two
terms in (\ref{eq:pbFG:monotonic}) is given by  by the relationship between $p+1$ and
$2q$. In fact when 
\begin{equation} \label{eq:nonlinearbalance_fg}
  p+1 > 2q 
\end{equation}
 the dissipative character of $f$ dominates the
explosive one of $g$; see below. 

\medskip

We now review some results already known for
(\ref{eq:pbFG:monotonic}). First, concerning local existence it was
shown in \cite{arrieta99:_parab} that considering initial data in $L^{r}(\Omega)$,
(\ref{eq:pbFG:monotonic}) is locally well posed provided 
\begin{equation}
\label{eq:Lr:growth:subcrit}
  p \leq p_{c} = 1 + {2r \over N}, \qquad   q \leq q_{c} = 1 + {r
    \over N}  
\end{equation}
with $q < q_{c} = 1 + r$, if $N=1$.

The numbers $p_{c}$ and $q_{c}$ above are the so-called critical exponents. The
problem (\ref{eq:pbFG:monotonic}) is said to be subcritical if $p<p_{c}$ and
$q<q_{c}$, and critical otherwise. See
\cite{Weissler1980:_Local-existence,Weissler1981:_Existence-and-n,brezis96} for
related results.

It was also proved in \cite{ARB04_nonwell_posed} that in general supercritical problems,
i.e. either $p>p_{c}$ or $q>q_{c}$, are ill posed.

As for the asymptotic behavior of solutions, (\ref{eq:pbFG:monotonic})
is studied    in \cite{Tajdine2001} 
in a subcritical $H^{1}(\Omega)$  setting (with suitable critical
exponents $p_{c}= 1 + {4 \over N-2}$ and $q_{c}= 1 + {2 \over N-2}$
for this space). Assuming the nonlinear balance condition
(\ref{eq:nonlinearbalance_fg}) and using a natural energy  estimate involving the
gradient, it was proved that (\ref{eq:pbFG:monotonic}) is dissipative
and the asymptotic behavior of solutions is described by a global
compact attractor in $H^{1}(\Omega)$ which typically has a precise
geometrical structure, since the problem has a gradient
structure. Note that in this case the energy acts as a Lyapunov
functional along solutions, which simplifies a lot the dynamics. 

On the other hand
if $p+1 < 2q$ the problem is not dissipative and it was shown in
\cite{Tajdine2001} that  there always exists solutions
that blow up in finite time. The case  when $p+1=2q$ depends on
the balance of the coefficients of the leading or even lower order
terms in $f$ and $g$. See  also 
\cite{arrieta08}. 

In \cite{arrieta00:_attrac_unif_bounds},  (\ref{eq:pbFG:monotonic}) was
studied in a subcritical regime and assuming a linear balance between
nonlinear terms. However for (\ref{eq:pbFG:monotonic}) under
assumptions (\ref{eq:mainHyp:fg}) and (\ref{eq:nonlinearbalance_fg})
no such linear balance holds.

In  \cite{Rodriguez-Bernal2002} the subcritical and
critical problems in $L^{r}(\Omega)$ (i.e, $1<p\leq p_C$, $1 < q\leq q_c$) were
considered. Again under the balance conditions (\ref{eq:mainHyp:fg})
and (\ref{eq:nonlinearbalance_fg}) it was shown that
(\ref{eq:pbFG:monotonic}) is dissipative and has well defined
asymptotic behavior in terms of a global compact attractor. In the
critical case a less conclusive result is obtained. See Theorem
\ref{th:attractor_subcritical} below for a precise statement. 

Note that for the setting in $L^{r}(\Omega)$, $r\neq 2$,  the  main
difficulty is to obtain gradient estimates to obtain compactness of
solutions. In this case no natural energy seems to be available.

\medskip

Observe that if we assume (\ref{eq:nonlinearbalance_fg}) and define  
\begin{equation} \label{eq:critical_space}
r_0 =\max\left\{\frac{N}{2}(p-1),N(q-1)\right\} = \frac{N}{2}(p-1)
\end{equation}
 then $r_{0} >1$ if $p> 1 + {2 \over N}$ and if  $r> r_0$ problem 
(\ref{eq:pbFG:monotonic}) is subcritical in $L^{r}(\Omega)$ while it
is critical if $r=r_{0}$, and supercritical if $1<r<r_{0}$. Thus if
$r\geq r_{0}$, given $u_0\in L^r(\Omega)$ we know that there exists a
unique global solution of problem (\ref{eq:pbFG:monotonic}) and
moreover there exists a global compact attractor.

Therefore in this paper we will assume hereafter that  
\begin{displaymath}
p> 1 + {2  \over N}  
\end{displaymath}
and  the main goal is the to prove that under
assumption (\ref{eq:nonlinearbalance_fg}), for the
supercritical case  $1<r < r_{0}$  still we have
global bounded solutions and a global compact attractor for
(\ref{eq:pbFG:monotonic}). Moreover this attractor coincides with the
one for $r>r_{0}$.  

Since the problem is supercritical in $L^{r}(\Omega)$ for $1<r<r_{0}$
we must then define suitable solutions of
(\ref{eq:pbFG:monotonic}). Then we need to obtain suitable and strong
estimates on the solutions which guarantee enough compactness (or
smoothing) to  prove the existence of an attractor. This estimates
will show that solutions of (\ref{eq:pbFG:monotonic}) enter into
spaces in which the problem is subcritical and then they are attracted
to the attractor of the subcritical case. As before the main
difficulty is to obtain gradient estimates on the solutions.

Our last contribution is to show that within the global attractor
there exist extremal equilibria as in
\cite{Rodr'iguez-Bernal2008}. These equilibria have the property that
the asymptotic behavior of any solutions lies below the maximal one
and above the minimal one. In particular any other equilibria lies in
between these two. Also the maximal equilibria is order stable from
above and the minimal one is order stable from below. Hence, they are
the ``caps'' of the global attractor. So far such equilibria have been
shown to exists in problems like (\ref{eq:pbFG:monotonic}) but in which a
suitable linear balance as in \cite{arrieta00:_attrac_unif_bounds} holds; see
\cite{Rodr'iguez-Bernal2008}. For (\ref{eq:pbFG:monotonic}) under
assumptions (\ref{eq:mainHyp:fg}) and (\ref{eq:nonlinearbalance_fg})
no such linear balance holds and \cite{Rodr'iguez-Bernal2008} does not
apply.

The paper is organized as follows. In Section
\ref{sec:known-results-subcr} we recall known  results  concerning existence of
solutions of (\ref{eq:pbFG:monotonic}) in subcritical or critical
cases in $L^r(\Omega)$. Observe that these  results only take into
account the growth of nonlinear terms and not the signs of $f$ and
$g$. Then, assuming (\ref{eq:mainHyp:fg}) and
(\ref{eq:nonlinearbalance_fg}) we also recall the results  on the
asymptotic behavior of solutions  of (\ref{eq:pbFG:monotonic}) in
subcritical or critical cases; see Theorem
\ref{th:attractor_subcritical} below.  

Then in Section \ref{sec:local-exist-lr-supercrit} we construct
suitable  solutions of problem
(\ref{eq:pbFG:monotonic}), assumed (\ref{eq:mainHyp:fg}) and
(\ref{eq:nonlinearbalance_fg}),  starting at $u_0\in L^r(\Omega)$ for
$1<r<r_0$ with $r_{0}$ as in (\ref{eq:critical_space}),  so that 
the problem is supercritical in $L^r(\Omega)$. For this, we start by proving the
existence and uniqueness of solutions in $L^r(\Omega)$ for an approximated problem in which we
truncate the nonlinear term in the boundary in such a way that the resulting
problem is supercritical in the interior while subcritical on the
boundary, see Theorems \ref{thr:truncated_problem_well_posed} and
\ref{th:uniqueness:truncate}. Since the supercritical nonlinear term
has a good sign, we will be able to obtain  suitable uniform bounds 
for these solutions  that will allow us to  construct a solution of
(\ref{eq:pbFG:monotonic}).  See Propositions
\ref{prop:bounds_vK_Lsigma}, \ref{prop:regFromBounds} and
\ref{prop:unifBound:genSols}, 
Definition \ref{def:solution_supercritical} and Theorem
\ref{thr:continuity_at_t=0}. Observe we could only guarantee
uniqueness of solutions in the case of nonnegative initial data in
(\ref{eq:pbFG:monotonic}), see Proposition \ref{prop:positiveSolns}
but not for general sign changing solutions.  

Then in Section \ref{sec:asymptotic-behaviour} we use the  strong
smoothing estimates  obtained in Section
\ref{sec:local-exist-lr-supercrit} to show that all solutions of
(\ref{eq:pbFG:monotonic}), regularize into an space in which 
(\ref{eq:pbFG:monotonic}) is subcritical and then the asymptotic
behavior of solutions is described by the global attractor in Theorem
\ref{th:attractor_subcritical}. We also show the existence of the
extremal equilibria in the attractor as discussed above; see Theorem
\ref{thm:existExtremal:supcrit}. Note that the  result in Section \ref{sec:asymptotic-behaviour} improve
the known result in Theorem \ref{th:attractor_subcritical} for the
critical case $r=r_{0}$.

\section{Known results for the subcritical or critical cases} 
\label{sec:known-results-subcr}
\setcounter{equation}{0}

We recall now the results in \cite{arrieta99:_parab} concerning existence of
solutions in subcritical or critical cases in $L^r(\Omega)$. These results only
take into account the growth of nonlinear terms and not the signs of $f$ and
$g$. Thus, we assume momentarily that $f,g\in C^{1}(\R)$ and satisfy
\begin{displaymath}
  \limsup_{|s|\to \infty} \frac{|f^\prime(s)|}{|s|^{p-1}} <\infty
    \quad\mathrm{and}\quad
    \limsup_{|s|\to \infty} \frac{|g^\prime(s)|}{|s|^{q-1}} <\infty
\end{displaymath}
and  $p,q$ satisfy
(\ref{eq:Lr:growth:subcrit}), i.e. 
\begin{displaymath}
    p \leq p_{c} = 1 + {2r \over N}, \qquad   q \leq q_{c} = 1 + {r
    \over N}, \quad  1<r<\infty  . 
\end{displaymath}
We will also make use of some  Bessel spaces in $L^{r}(\Omega)$  of order $2\theta$,
$H_{r}^{2\theta}(\Omega)$; see
\cite{arrieta99:_parab,amann93:_nonhom,triebel95:_inter}. These spaces
are Bessel spaces associated to $\Delta$ with Neumann boundary 
conditions in $L^r(\Omega)$. Hence, if $2\theta > 1 + {1 \over r}$
these spaces incorporate Neumann boundary conditions.

From \cite{arrieta99:_parab}, for each $u_{0} \in L^{r}(\Omega)$, there
exist $R = R(u_{0}) >0$ and $\tau = \tau(u_{0}) >0$ such that for any $u_1 \in
L^{r}(\Omega)$ with $\|u_1 - u_0\|_{L^{r}(\Omega)} < R$ there exists a
solution of (\ref{eq:pbFG:monotonic}), with initial data $u_{1}$,
$u(\cdot;u_{1})$,  in
the sense that it is a continuous function $u:[0,\tau_0]\to
L^{r}(\Omega)$   with $u(0)= u_1$, such that $u\in
C([0,\tau];L^r(\Omega))\cap C((0,\tau]; H^{2\bar\epsilon}_r(\Omega))$
and $\sup_{t\in (0,\tau]} t^{\bar \epsilon}
\|u(t)\|_{H^{2\bar\epsilon}_r(\Omega)} < \infty$,  for some $\bar 
\epsilon>0$ and satisfies  
 the variation of constants formula
\begin{displaymath}
\label{eq:vcf}
  u(t;u_1) = S(t) u_1 
  + \int_0^{t} S(t-s)(-f_\Omega (u(s;u_1)) +g_\Gamma
  (u(s;u_1))\,\mathrm{d}s , \quad 0\leq t \leq \tau, 
\end{displaymath}
where $S(t)$ is the semigroup generated by $\Delta$ with Neumann boundary
conditions in $L^r(\Omega)$, $f_\Omega$ denotes the Nemitsky map of
$f$ acting on functions defined in $\Omega$, and $g_\Gamma$ the
Nemitsky  map of  $g$ acting on functions defined on $\Gamma$.

This is the so called
$\overline\epsilon-$regular solution of 
(\ref{eq:pbFG:monotonic}) starting at $u_1$.  This solution is unique in the
class $C([0,\tau];L^r(\Omega))\cap C((0,\tau];H^{2\bar \epsilon}_r(\Omega))$
and, by a bootstrapping argument, it  is classical for $t>0$.

In addition, this solution satisfies, for some $\underline\gamma >
\overline\epsilon$ and for all $0<\theta< \underline\gamma$,
\begin{equation} \label{epsregsol1}
u \in C((0,\tau]; H_{r}^{2\theta}(\Omega)), \quad \sup_{t\in (0,\tau]}
t^\theta
\|u(t)\|_{H_{r}^{2\theta}(\Omega)} \leq M(R,\tau), \quad
t^\theta \|u(t)\|_{H_{r}^{2\theta}(\Omega)} {\buildrel t\to
  0^+\over\longrightarrow } 0 . 
\end{equation}
 Moreover, if $u_1,
v_1\in B_{L^{r}(\Omega)}(u_0, R)$ the following holds true for $t\in
(0,\tau]$, and $0\le \theta\leq \theta_0<\underline \gamma$, 
\begin{equation} \label{epsregsol2}
\sup_{t\in (0,\tau]} t^\theta \|u(t; u_1) - u(t; v_1)\|_{H_{r}^{2\theta}(\Omega)}
\leq C(\theta_0,\tau) \|u_1 - v_1\|_{L^{r}(\Omega)}.  
\end{equation}
Note that we can always take $\underline \gamma \geq 1/2$ which allows to
perform a bootstrap argument to prove that solutions become classical for
positive times.

If both of the nonlinearities are subcritical, i.e. $p<p_{c}$ and
$q<q_{c}$ in (\ref{eq:Lr:growth:subcrit}),  then $R$ can be taken arbitrarily
large and so, the existence time can be taken uniform on bounded sets of
$L^{r}(\Omega)$. As a consequence, and following a standard prolongation
argument, when $f$ and $g$ are subcritical in $L^r(\Omega)$, if the solution
exists up to a maximal time $T<\infty$ then $\lim_{t\rightarrow T}
{\left\|u(t)\right\|}_{L^{r}(\Omega)} = \infty$. However, when $f$ or $g$ are
critical, if $T<\infty$ then $\lim_{t\rightarrow T}
{\left\|u(t)\right\|}_{H^{\delta}_r(\Omega)} = \infty$ for any
$\delta>0$. 
Therefore, in the subcritical case to prove global existence it is enough to
obtain bounds on the $L^{r}(\Omega)$--norm of the solution while in the critical
case stronger estimates must be obtained.

On the other hand, if we assume now (\ref{eq:mainHyp:fg}), that is, 
\begin{displaymath}
  pc_f |s|^{p-1} - A_0\leq f^\prime(s) \leq pC_f|s|^{p-1} + A_1
\quad \mathrm{and}\quad 
  qc_g |s|^{q-1} - B_0\leq g^\prime(s) \leq qC_g|s|^{q-1} + B_1, 
\end{displaymath}
for $s\in \R$ and some $c_f, c_g ,C_f,C_g>0$, $A_0, A_1,B_0,B_1>0$,  and
(\ref{eq:nonlinearbalance_fg}), i.e. 
\begin{displaymath}
    p+1 > 2q , 
\end{displaymath}
we can obtain
information on the  asymptotic behavior of the solutions of the
problem. 

The next result summarizes the results in 
\cite{Rodriguez-Bernal2002}. Note that
condition (\ref{eq:nonlinearbalance_fg}) implies that the dissipative
character of $f$ dominates the explosive nature of $g$. As a result of
this nonlinear balance,  (\ref{eq:pbFG:monotonic}) is dissipative as the
next theorem shows. 

\begin{theorem} 
\label{th:attractor_subcritical}

Assume that $f$ and $g$ satisfy
  (\ref{eq:mainHyp:fg}) and 
  (\ref{eq:nonlinearbalance_fg}) and define $r_0$ as in
  (\ref{eq:critical_space}). Then we have, 

  \noindent i) Problem (\ref{eq:pbFG:monotonic}) is well-posed in $L^r(\Omega)$
  for any $r\geq r_0$,   and the
  solutions are globally defined, and classical for $t>0$.

\noindent ii) For $r\geq r_0$, there exists an absorbing ball in
$L^{r}(\Omega)$ and the orbit of any bounded set of
$L^{r}(\Omega)$ is bounded in $L^{r}(\Omega)$, for $t\geq 0$. 

For $r=r_{0}$, orbits of compact sets in     $L^{r_0}(\Omega)$
remain compact in $L^{r_0}(\Omega)$.

\noindent iii) If $r>r_{0}$,  (\ref{eq:pbFG:monotonic}) has a compact
global attractor $\mathcal{A}$ in $L^r(\Omega)$ which attracts bounded
sets     of $L^r(\Omega)$.  

For $r=r_{0}$,  there exists a maximal, compact, invariant and
connected set $\mathcal{A}$ in $L^{r_0} (\Omega)$ which attracts a
neighborhood of each initial data in     $L^{r_0} (\Omega)$ (and, in
particular, compact sets of $L^{r_0} (\Omega)$). 
    
For $r\geq r_{0}$ the attractor can be described as the unstable set of
the set of equilibria of (\ref{eq:pbFG:monotonic}), $E$,  which is nonempty;
that is 
\begin{displaymath}
\mathcal{A} = W^{u}(E) .   
\end{displaymath}

\noindent iv) For $r\geq r_{0}$, the attractor $\mathcal{A}$ belongs to
$H^{\alpha}_{s}(\Omega)$ and it attracts in the norm of
$H^{\alpha}_{s}(\Omega)$ for any $s\geq 1$ and $0\leq\alpha<1 + \frac{1}{s}$ and
in $C^\beta(\overline\Omega)$ for any $0\leq\beta<1$. 

In particular, if $r>r_{0}$ there is an absorbing set in
$H^{\alpha}_{s}(\Omega)$  and, for every $\eps>0$,  the orbit of bounded sets in
$L^{r}(\Omega)$ is bounded in $H^{\alpha}_{s}(\Omega)$ for $t\geq
\eps$. 
\end{theorem}

Observe that this result is proved in \cite{Rodriguez-Bernal2002}
using in a critical way the energy estimate (\ref{eq:energy_NL})
below, which is the only estimate available in a non Hilbertian
setting. In \cite{Tajdine2001}, (\ref{eq:pbFG:monotonic}) is studied
in a subcritical $H^{1}(\Omega)$  setting (with suitable critical
exponents $p_{c}= 1 + {4 \over N-2}$ and $q_{c}= 1 + {2 \over N-2}$
for this space) using a different energy estimate involving the
gradient. For the setting in $L^{r}(\Omega)$, $r\neq 2$,  the  main
difficulty is to obtain gradient estimates to obtain compactness. 

Also observe that the estimates in \cite{Rodriguez-Bernal2002} leading
to Theorem \ref{th:attractor_subcritical} are not known to be uniform
with respect to certain classes of  nonlinear terms $f$ and $g$. Should this be true,
some arguments below in Section \ref{sec:local-exist-lr-supercrit}
could be made simpler, see Remark \ref{rem:seria_masfacil}.  

Below we present some of the basic tools needed to prove  Theorem
\ref{th:attractor_subcritical}, see \cite{Rodriguez-Bernal2002}. First
we have the following  Poincar\'e lemma. 

\begin{lemma}
  \label{lemma:PoincareL1}
  There exists a constant $c_0(\Omega)$ such that for any $\varphi \in
  W^{1,1}(\Omega)$
  \begin{displaymath} \label{eq:PoincareL1}
    \left\|\varphi - \frac{1}{|\Gamma|}\int_\Gamma \varphi\right\|_{L^1(\Omega)}
    \leq c_0(\Omega) \|\nabla \varphi\|_{L^1(\Omega)}.
  \end{displaymath}
\end{lemma}

Another key tool to prove Theorem \ref{th:attractor_subcritical} is the
following estimate that holds for any suitable smooth solution of
(\ref{eq:pbFG:monotonic}). This result explains in a precise form the nonlinear
balance (\ref{eq:nonlinearbalance_fg}) between nonlinear terms in the problem
(\ref{eq:pbFG:monotonic}). 
Note that the result in \cite{Rodriguez-Bernal2002} is adapted below
to (\ref{eq:pbFG:monotonic}) with $f$ and $g$ satisfying
(\ref{eq:mainHyp:fg}) and (\ref{eq:nonlinearbalance_fg}). We include
the proof since it will be important in what 
follows.

\begin{proposition}
\label{prop:norm:ineq:bounds}
For a sufficiently smooth solution of (\ref{eq:pbFG:monotonic})  with $f$ and
$g$ satisfying (\ref{eq:mainHyp:fg}) and
(\ref{eq:nonlinearbalance_fg}) we have
\begin{equation} \label{eq:energy_NL} 
{\frac{1}{\sigma}} {\mathrm{d} \over \mathrm{d}t} \left\| u(t)
\right\|^{\sigma}_{L^{\sigma}(\Omega)} +
{\frac{2(\sigma-1)}{\sigma^2}} \int_{\Omega}
\left| \nabla \left( |u|^{\sigma/2} \right) \right|^2 + A \int_\Omega |u|^{\sigma+p-1} \leq B 
\end{equation} 
with $A$  depending on the constants appearing on
(\ref{eq:mainHyp:fg}) but not on $\sigma$ and $B>0$ depending also on
$\sigma$. 

\end{proposition}
\begin{proof}{}
For a sufficiently smooth solution of (\ref{eq:pbFG:monotonic}), 
multiplying (\ref{eq:pbFG:monotonic}) by  $|u|^{\sigma-2}u$ and
integrating by parts, we get
\begin{displaymath} 
{\frac{1}{\sigma}} {\mathrm{d} \over\mathrm{d}t} \left\| u(t)
\right\|^{\sigma}_{L^{\sigma}(\Omega)} +
{\frac{4(\sigma-1)}{\sigma^2}} \int_{\Omega}
\left| \nabla \left( |u|^{\sigma/2} \right) \right|^2 + \int_\Omega
f(u)|u|^{\sigma-2} u - \int_{\Gamma} g(u)|u|^{\sigma-2} u =0.
\end{displaymath}
The last two terms can be rewritten as
$$
\int_\Omega\left(f(u) |u|^{\sigma-2} u - 
{|\Gamma|\over|\Omega|} g(u) |u|^{\sigma-2} u\right) + 
{|\Gamma|\over|\Omega|} \int_\Omega
\left(g(u) |u|^{\sigma-2} u-\frac{1}{|\Gamma|}\int_\Gamma
 g(u) |u|^{\sigma-2} u\right)
$$
and  Lemma \ref{lemma:PoincareL1}  gives, for some $c(\Omega)$ 
$$
\left|{|\Gamma| \over| \Omega|} \int_\Omega \left(g(u)
|u|^{\sigma-2} u- \frac{1}{|\Gamma|}\int_\Gamma g(u) |u|^{\sigma-2} u\right) \right|
\leq  c(\Omega) \|\nabla(g(u) |u|^{\sigma-2} u)\|_{L^1(\Omega)} .
$$
Taking derivatives and arranging terms, the right hand side above  can
be written as
$$
 c(\Omega) \left\| \left({2 \over \sigma} g'(u)  u + {2
(\sigma-1) \over \sigma} g(u) \right) |u|^{\sigma/2-1} \left| \nabla |u|^{\sigma/2}
\right| \right\|_{L^1(\Omega)}
$$
which can be bounded by
$$
\eps\|\nabla |u|^{\sigma/2} \|^2_{L^2(\Omega)}+
{c^2(\Omega)\over 4\eps}
\left\| \left({2 \over \sigma} g'(u)  u + {2
(\sigma-1) \over \sigma} g(u) \right) |u|^{\sigma/2-1} \right\|^{2}_{L^2(\Omega)}
$$
for any $\eps > 0$. Therefore, we get
\begin{equation} \label{eq:energyLsigma} 
{\frac{1}{\sigma}} {d \over dt} \left\| u(t)
\right\|^{\sigma}_{L^{\sigma}(\Omega)} +
\left( {\frac{4(\sigma-1)}{\sigma^2}} - \eps\right) \int_{\Omega}
\left| \nabla \left( |u|^{\sigma/2} \right) \right|^2 + \int_\Omega
H_{\sigma}(u) |u|^{\sigma-2} \leq 0 
\end{equation}
where
\begin{displaymath}
\label{eq:def:Hsigma}
  H_{\sigma}(u)  = f(u)u - {|\Gamma|\over|\Omega|}g(u)u-
{c^2(\Omega)\over \eps \sigma^{2}} \left( g'(u)  u +
(\sigma-1)  g(u)\right)^2.
\end{displaymath}

From (\ref{eq:mainHyp:fg}) we have that, 
\begin{displaymath}
  f(u)u \geq D_{0} |u|^{p+1} - D_1, \quad 
  g(u) u \leq D_{2} |u|^{q+1} + D_1
\end{displaymath}
\begin{displaymath}
  (g^\prime (u) u + (\sigma -1) g(u))^2
\leq 2( |g^\prime (u)|^2 u^2 + (\sigma -1)^{2}|g(u)|^2)
\leq D_3(|u|^{2q} + 1)
\end{displaymath}
and so, 
\begin{displaymath}
  H_{\sigma}(u) \geq D_{0} |u|^{p+1} - D_2\frac{|\Gamma|}{|\Omega|}|u|^{q+1}
- D_3 \frac{c^2(\Omega)}{\eps \sigma^2} |u|^{2q} - D_4 
\end{displaymath}
for some $D_1,\ldots, D_4>0$ where $D_{0}, D_{1}, D_{2}$ do not depend on
$\sigma$. Since $p+1 > 2q$ then we 
have  $H_{\sigma}(u) \geq D_{5}  |u|^{p+1} - D_{6}$ for some positive
constants $D_{5} < D_{0}$  which depends only on $p,q$ and
the constants in (\ref{eq:mainHyp:fg}) and $D_{6}$ depends also in
$\sigma$ and $\eps$. From this we get  
\begin{displaymath} 
\label{eq:importantProperty:Hsigma}
  H_{\sigma}(u) |u|^{\sigma-2} \geq A |u|^{\sigma+p-1} -B, \quad u \in \R.
\end{displaymath}
with $A$ not depending on $\sigma$.  Then taking $\eps=
\frac{2(\sigma-1)}{\sigma^2}$  in (\ref{eq:energyLsigma}), we get
(\ref{eq:energy_NL}). 
\end{proof}

Also, we will use below the following  Lemma. 

\begin{lemma}
For any smooth enough function in
$\overline{\Omega}$, we have, for any $\delta>0$ 
\begin{equation} \label{eq:estimate_Lr_trace}
   \int_\Gamma |\varphi|^r  \leq \delta \|\nabla
  (|\varphi|^{r/2})\|_{L^2(\Omega)}^2 +
  c(\Omega,\Gamma,\delta)\|\varphi\|_{L^r(\Omega)}^r 
\end{equation} 
\end{lemma}
\begin{proof}
We know that for any $\delta>0$, there exists $C_\delta>0$ such that
\begin{displaymath}
  \int_{\Gamma}|\xi|^2\,\mathrm{d}x \leq \delta 
  \int_{\Omega} |\nabla \xi|^2\,\mathrm{d}x
  + C_\delta\int_{\Omega}  |\xi|^2\,\mathrm{d}x
\end{displaymath}
for any $\xi \in H^1(\Omega)$.  Taking $\xi = |\varphi|^{r/2}$  we
obtain (\ref{eq:estimate_Lr_trace}). 
\end{proof}

\section{Existence in $L^r(\Omega)$, $1<r<r_0$}
\label{sec:local-exist-lr-supercrit}
\setcounter{equation}{0}

The goal of the section is to prove the existence of a solution of problem
(\ref{eq:pbFG:monotonic}), assumed (\ref{eq:mainHyp:fg}) and
(\ref{eq:nonlinearbalance_fg}),  starting at $u_0\in L^r(\Omega)$ for
$1<r<r_0$ with $r_{0}$ as in (\ref{eq:critical_space}),  so that 
the problem is supercritical in $L^r(\Omega)$. For this, we start by proving the
existence of solutions in $L^r(\Omega)$ for an approximated problem in which we
truncate the nonlinear term in the boundary in such a way that the resulting
problem is supercritical in the interior while subcritical on the
boundary. Since the supercritical nonlinear term has a good sign,
later, suitable uniform bounds for these solutions will allow us to 
construct a solution of (\ref{eq:pbFG:monotonic}).

\subsection{A supercritical truncated problem}

Notice that for $f$ satisfying (\ref{eq:fg:monotonic}) there exists $L>0$ such
that
\begin{equation}
  \label{eq:fg:monotonic}
  f'(s) \geq -L 
\qquad  \textrm{for all}\     s\in \R.
\end{equation}
In fact, it is enough to take $L = A_0$ with $A_0$ from (\ref{eq:mainHyp:fg}).

Also notice that we can construct functions $g_K$ satisfying the
following properties
\begin{enumerate}
\item $sg_K(s)$ is \emph{increasing} in $K$, i. e., for any $K_2 > K_1 >0$,
  \begin{equation} \label{eq:gKincreases_in_K}
    sg_{K_1}(s) \leq s g_{K_2}(s)\leq sg(s) \quad \mbox{for all $s\in
      \R$} 
  \end{equation}
 and
\item $g_K$ is $C^1(\R)$ and
  \begin{equation}
    \label{eq:cutGK}
-B_{0} \leq g_K'(s) \leq K \qquad \textrm{for all}\ s\in \R,
 \end{equation}
with $B_{0}$ as in (\ref{eq:mainHyp:fg}). 

\item $g_{K}$ coincides with $g$ in an interval $I_{K}= [a_{K}, b_{K}]$
with 
\begin{equation} \label{eq:increasing_interval}
  a_{K} \to -\infty, \quad  b_{K} \to \infty \quad \mbox{as $K \to \infty$}. 
\end{equation}
\end{enumerate}

\medskip

Given $K>0$   we consider
the following truncated problem  
\begin{equation}  \label{eq:truncated_problem}
  \left\{\begin{array}{rclcl}
    u_t  - \Delta u + f(u) & =
    & 0 & \mathrm{in} & \Omega\\
    \displaystyle\frac{\partial u}{\partial \vec{n}}&=& g_K(u) &
    \mathrm{on} & \Gamma\\
    u(0) &=& u_0
  \end{array}\right.
\end{equation}
with initial data in $L^{r}(\Omega)$. 
Notice that the reaction terms in (\ref{eq:truncated_problem}) satisfy
(\ref{eq:mainHyp:fg}) and (\ref{eq:nonlinearbalance_fg}) with $q=1$
and 
\begin{equation} 
\label{eq:linear_bounds_f_g}
sf(s)\geq -|f(0)||s|-L|s|^2, \quad    sg_K(s) \leq |g(0)||s| + K|s|^2   
 \qquad \textrm{for all}\ s\in \R 
\end{equation}

 The next results show that (\ref{eq:truncated_problem}) is globally well posed
 in $L^{r}(\Omega)$. Since (\ref{eq:truncated_problem}) is
 supercritical in $L^{r}(\Omega)$ because $1<r<r_{0}$, even
 local existence does not follow from previous results in Section
 \ref{sec:known-results-subcr}.

\begin{theorem} \label{thr:truncated_problem_well_posed}

Let $1 < r < r_{0}$ and  $u_{0} \in L^{r}(\Omega)$. Then there
exists a function  $v^{K}$, such that for every $T>0$, 
\begin{displaymath}
v^{K} \in
C([0,T];L^r(\Omega))\cap L^r((0,T);L^r(\Gamma)) , \quad    v^{K}(0) = u_{0},
\end{displaymath}  
and  $v^{K} \in C([\eps, T]\times \overline \Omega)$, for every
$\eps>0$,  with  
\begin{displaymath}
  |v^K(t,x)| \leq C(T,K) +    t^{-\frac{N}{2r}} C(T,K)
  \|u_{0}\|_{L^r(\Omega)} ,
  \qquad 0 <t \leq T  \quad  \textrm{for all  } x \in \overline{\Omega} ,
\end{displaymath}
for some constant $C(T,K)\geq 0$, which is a  global solution of
(\ref{eq:truncated_problem}) in the sense that for all $t\geq 0$
satisfies the  
variation of constants formula  
\begin{displaymath} 
    v^{K}(t) = S(t)u_0+ \int_0^t S(t-s)\left( -f_{\Omega}(v^{K}(s)) +
    (g_K)_{\Gamma}(v^{K}(s)) \right) \,\mathrm{d}s 
\end{displaymath}
where $S(t)$ denotes the semigroup generated by $\Delta$ with Neumann boundary
conditions in $L^r(\Omega)$.

Moreover, $v^K\in C^1((0,\infty); C^2(\overline\Omega))$ and is a
classical solution for $t>0$ of (\ref{eq:truncated_problem}). 

\end{theorem}
\begin{proof}{}
We proceed in several steps. 
  
\medskip 
\noindent {\bf Step 1. Approximate the initial data} 

Let $u_0^n \in L^\sigma(\Omega)$ for some $\sigma> r_0$, with $r_0$ as in
(\ref{eq:critical_space}), such that $u_0^n \to u_0$ in $L^r(\Omega)$ as
$n\to\infty$ and consider the solutions of
\begin{equation}
  \label{eq:pbFG:monotonic:appr}
  \left\{\begin{array}{rclcl}
    u_t  - \Delta u + f(u) & =
    & 0 & \mathrm{in} & \Omega\\
    \displaystyle\frac{\partial u}{\partial \vec{n}}&=& g_K(u) &
    \mathrm{on} & \Gamma\\
    u(0) &=& u^n_0
  \end{array}\right.
\end{equation}
as in Section \ref{sec:known-results-subcr}, which we will denote by $u_n^K(t)$.
Since $f, g_{K}$  satisfy
(\ref{eq:mainHyp:fg}) and (\ref{eq:nonlinearbalance_fg}) with $q=1$,
by part i) in  Theorem \ref{th:attractor_subcritical}, these solutions
are global and  classical  for $t>0$.

Denote $v^K_{n,m}(t) = u_n^K(t) - u_m^K(t)$. Subtracting equations
for $u_n^K$ and $u_m^K$ and multiplying by
$|v^K_{n,m}(t)|^{r-2}v^K_{n,m}(t)$ we have
\begin{eqnarray*}
  \frac{1}{r}\frac{\mathrm{d}}{\mathrm{d}t} \|v^K_{n,m}(t)\|^r_{L^r} +
 c \|\nabla |v^K_{n,m}(t)|^{r/2}\|^2_{L^2(\Omega)} 
  &+& \int_{\Omega}
  \left[f(u^K_n(t)) - f(u^K_m(t))\right] |v^K_{n,m}(t)|^{r-2}v^K_{n,m}(t) \\
  &\leq &  \int_\Gamma \left[g_K(u^K_n(t)) - g_K(u^K_m(t))\right]
  |v^K_{n,m}(t)|^{r-2}v^K_{n,m}(t).
\end{eqnarray*}
Now, observe that from (\ref{eq:fg:monotonic}), (\ref{eq:cutGK})
\begin{displaymath}
   \int_{\Omega}
  \left[f(u^K_n) - f(u^K_m)\right]
  |v^K_{n,m}|^{r-2}v^K_{n,m}\geq -L \|v^K_{n,m}\|^r_{L^r(\Omega)}
\end{displaymath}
and
\begin{displaymath}
  \int_\Gamma \left[g_K(u^K_n) - g_K(u^K_m)\right]
  |v^K_{n,m}|^{r-2}v^K_{n,m} \leq K \int_\Gamma |v^K_{n,m}|^r.
\end{displaymath}

Hence, from (\ref{eq:estimate_Lr_trace}),   for any $\delta>0$ 
\begin{displaymath}
   \int_\Gamma \left[g_K(u^K_n) - g_K(u^K_m)\right]
  |v^K_{n,m}|^{r-2}v^K_{n,m} \leq \delta \|\nabla
  (|v_{n,m}^K|^{r/2})\|_{L^2(\Omega)}^2 +
  c(K,\delta)\|v_{n,m}^K\|_{L^r(\Omega)}^r . 
\end{displaymath}
Thus, with a suitable choice of $\delta$   we have
\begin{equation} \label{eq:Lipschitz_truncated_problem}
\frac{\mathrm{d}}{\mathrm{d}t} \|v^K_{n,m}(t)\|^r_{L^r(\Omega)} +
   c_1 \|\nabla (|v^K_{n,m}(t)|^{r/2})\|^2_{L^2(\Omega)}
   \leq C(K)   \|v^K_{n,m}(t)\|^r_{L^r(\Omega)}.
\end{equation}

Given $T>0$, by Gronwall's Lemma, from
(\ref{eq:Lipschitz_truncated_problem}) we have that for any $0\leq
t\leq T$, 
\begin{displaymath}
  \|v^K_{n,m}(t)\|^r_{L^r(\Omega)} + c_1\int_0^t \|\nabla
  (|v^K_{n,m}(s)|^{r/2})\|^2_{L^2(\Omega)}\,\textrm{d}s
\leq  C(K,T)
  \|v^K_{n,m}(0)\|^r_{L^r(\Omega)} \to 0 \quad \textrm{as} \quad n,m\to\infty
\end{displaymath}
and so, $u_n^K$ is a Cauchy sequence in $C([0,T];L^r(\Omega))$. Also
using  (\ref{eq:estimate_Lr_trace}) 
\begin{eqnarray*}
  \int_{0}^{T} \int_\Gamma |v^K_{n,m}(t)|^r  \mathrm{d}t &\leq&  \int_{0}^{T} \|\nabla
  (|v^K_{n,m}(t)|^{r/2})\|_{L^2(\Omega)}^2  \mathrm{d}t +
  c(\Omega,\Gamma) \int_{0}^{T} \|v^K_{n,m}(t)\|_{L^r(\Omega)}^r   \mathrm{d}t
  \\ & \leq&
  C(K,T)   \|v^K_{n,m}(0)\|^r_{L^r(\Omega)}  \to 0 \quad \textrm{as}
  \quad n,m\to\infty 
\end{eqnarray*} 
and then $u_n^K$ is also a  Cauchy sequence in
$L^{r}((0,T);L^r(\Gamma))$. 

Hence, there exists $v^K\in C([0,T];L^r(\Omega))\cap
L^r((0,T);L^r(\Gamma))$ such that
\begin{equation} \label{eq:convergence_2_vK_LrOmega} 
  \sup_{t\in [0,T]} \|u_n^K(t) - v^K(t)\|_{L^r(\Omega)} \to 0 \quad
  \textrm{as} \quad n\to\infty
\end{equation}
\begin{equation} \label{eq:convergence_2_vK_LrGamma} \int_{0}^{T} \int_{\Gamma}
  |u_n^K(t,x) - v^K(t,x)|^r \, \mathrm{d}x\mathrm{d}t \to 0 \quad \textrm{as}
  \quad n\to\infty .
\end{equation} 
In particular, $u_n^K(t,x) \to v^K(t,x)$ as
$n\to\infty$ a.e. for  $(t,x)\in[0,T]\times\Gamma$. 

Also it is easy to see that $v^{K}$ does not depend on the sequence of
initial data, but only on $u_{0} \in L^{r}(\Omega)$.

\medskip 
\noindent {\bf Step 2. $L^\infty$-bound for the approaching sequence} 

Let us  show now that the sequence $u_n^K(t)$ is uniformly bounded in 
$L^\infty(\Omega)$ with respect to $n$, for $0<\eps\leq t\leq T$. 

For this, we will use the auxiliary problem
\begin{equation}
  \label{eq:pb:lin:bound}
   \left\{\begin{array}{rclcl}
    U_t  - \Delta U  & =
    & LU + A & \mathrm{in} & \Omega\\
    \displaystyle\frac{\partial U}{\partial \vec{n}}&=& K U +D &
    \mathrm{on} & \Gamma\\
    U(0) & &  \textrm{given in   }L^{r}(\Omega) 
  \end{array}\right.
\end{equation}
with $A= |f(0)|$ and $D= |g(0)|$. Denote by $U^{n}(t,x)$ the solution
of (\ref{eq:pb:lin:bound}) with initial data $|u_{0}^{n}|$ and by $U(t,x)$ the solution
of (\ref{eq:pb:lin:bound}) with initial data $|u_{0}|$. 

Now, using the variation of constants formula in
(\ref{eq:pb:lin:bound}) 
\begin{displaymath}
  U^n(t) = \Phi(t) + U^n_h(t), \quad  U(t) = \Phi(t) + U_h(t)
\end{displaymath}
where $U^n_h(t), U_h(t)$ are  the solutions of the homogeneous problem resulting from
taking $A=D=0$ in (\ref{eq:pb:lin:bound}) and initial data
$|u_{0}^{n}|$ and $|u_{0}|$ respectively, i.e, the solution of
\begin{displaymath}
   \left\{\begin{array}{rclcl}
    U_t  - \Delta U  & =
    & LU & \mathrm{in} & \Omega\\
    \displaystyle\frac{\partial U}{\partial \vec{n}}&=& K U &
    \mathrm{on} & \Gamma\\
    U(0) & =&  |u_0^n| \quad (\mathrm{or}\ |u_0|)
  \end{array}\right.
\end{displaymath}
 and $\Phi(t)$ is the unique solution of
problem (\ref{eq:pb:lin:bound}) with $U(0)=0$ (which does not depend on
$u_{0}^{n}$ or $u_{0}$), i.e,
\begin{displaymath}
   \left\{\begin{array}{rclcl}
    U_t  - \Delta U  & =
    & LU + A & \mathrm{in} & \Omega\\
    \displaystyle\frac{\partial U}{\partial \vec{n}}&=& K U +D &
    \mathrm{on} & \Gamma\\
    U(0) & =&  0.
  \end{array}\right.
\end{displaymath}

Notice that the homogeneous problem above is a linear heat equation
with Robin boundary conditions. Therefore,
standard regularity theory implies that $U^n_h $, $U_h$ and $\Phi$
are classical for $t>0$,
\begin{displaymath}
  \|U^n(t)\|_{L^\infty(\Omega)} \leq C(T,K) +  
  t^{-\frac{N}{2r}} C(T,K) \|u_{0}^{n}\|_{L^r(\Omega)}, \qquad 0 <t
  \leq T   ,
\end{displaymath}
and 
\begin{displaymath}
  U^{n}(t,x) \to U(t,x) \quad \textrm{in} \  C^{k}([\eps,T]\times
  \Omega) \cap C([0,T]; L^{r}(\Omega))  
\end{displaymath}
for any $\eps >0$, $k\in \N$, since $u_0^n \to u_0$ in $L^r(\Omega)$ as
$n\to\infty$.

Also, $U^n_h (t)\geq 0$, $U_h(t)\geq 0$ and $\Phi(t)\geq 0$
for all $t\geq 0$ since $A,D \geq 0$ and have nonnegative initial data.

Observe now that $U^{n}(t,x)$ 
is  a supersolution  of problem
(\ref{eq:pbFG:monotonic:appr}) since $f,g_K$
satisfy (\ref{eq:linear_bounds_f_g}).
Hence 
\begin{displaymath} 
  |u_n^K(t,x)| \leq  U^{n}(t,x) \leq C(T,K) +    t^{-\frac{N}{2r}}
  C(T,K) \|u_{0}^{n}\|_{L^r(\Omega)}, \qquad 0 <t \leq T 
\quad \textrm{a.e.\ in  }\Omega, 
\end{displaymath}
and so $\|u^K_n(t)\|_{L^\infty(\Omega)}\leq
C(\eps,T,K,\|u_{0}^{n}\|_{L^r(\Omega)})$ for all $n\geq
1$ and $\eps\leq t \leq T$. Also since $u_{n}^K$ is a classical
solution and $U^{n}$ is also smooth we have that, up to the boundary, 
\begin{equation} \label{eq:bound_Linfty_u_n}
  |u_n^K(t,x)| \leq U^{n}(t,x) \leq  C(T,K) +    t^{-\frac{N}{2r}}
  C(T,K) \|u_{0}^{n}\|_{L^r(\Omega)}, \qquad 0 <t \leq T  
\quad \textrm{for all  } x \in \overline{\Omega}.  
\end{equation}

Now, since $u_0^n \to u_0$ in $L^r(\Omega)$ as $n\to\infty$ and using  the
convergence $U^{n}(t,x) \to U(t,x)$ and  $u_n^K(t,x) \to  v^K(t,x)$ obtained above,
(\ref{eq:convergence_2_vK_LrOmega}),
(\ref{eq:convergence_2_vK_LrGamma}),  we get  
\begin{equation}
  \label{eq:bound:vK:Linfty}
  |v^K(t,x)| \leq U(t,x) \leq C(T,K) +    t^{-\frac{N}{2r}} C(T,K)
  \|u_{0}\|_{L^r(\Omega)} ,
  \qquad 0 <t \leq T  \quad  \textrm{for all  } x \in \overline{\Omega}. 
\end{equation}
Notice that estimates (\ref{eq:bound_Linfty_u_n}) and (\ref{eq:bound:vK:Linfty})
are valid up to the boundary.

Finally, observe that the bounds above  and
(\ref{eq:convergence_2_vK_LrOmega}), (\ref{eq:convergence_2_vK_LrGamma}) imply
that for any $\eps >0$ and any $r\leq s < \infty$,
\begin{equation} \label{eq:convergence_2_vK_LqOmega} 
  \sup_{t\in [\eps,T]} \|u_n^K(t) - v^K(t)\|_{L^s(\Omega)} \to 0 \quad
  \textrm{as} \quad n\to\infty
\end{equation}
\begin{equation} \label{eq:convergence_2_vK_LqGamma}  
  \int_{\eps}^{T} \int_{\Gamma} |u_n^K(t,x) - v^K(t,x)|^s \,
  \mathrm{d}x\mathrm{d}t \to 0 \quad
  \textrm{as} \quad n\to\infty . 
\end{equation}

\medskip 
\noindent {\bf Step 3. The limit is a solution of (\ref{eq:truncated_problem})}

First, assume $0<\eps<t<T$. Taking now $\phi\in H^{2}_{r^\prime}(\Omega)$, with 
$\frac{\partial\phi}{\partial\vec{n}}=0$ on $\Gamma$, where $r^\prime$ is the conjugate
of $r$, i. e., $\frac{1}{r}+ \frac{1}{r^\prime}=1$, we have from
(\ref{eq:pbFG:monotonic:appr}) 
\begin{displaymath}
   \frac{\mathrm{d}}{\mathrm{d}t}  \int_\Omega u_n^K \phi +
   \int_\Omega u_n^K (-\Delta \phi) +\int_\Omega  f(u^K_n) \phi 
=\int_\Gamma g_K(u_n^K)\phi .  
\end{displaymath}

Now, using  the uniform bounds in (\ref{eq:bound_Linfty_u_n}),
(\ref{eq:bound:vK:Linfty}) and the convergence in
(\ref{eq:convergence_2_vK_LqOmega}),
(\ref{eq:convergence_2_vK_LqGamma}), and the growth of $f$, we have for any
$r\leq s<\infty$,
\begin{displaymath}
   f(u^K_n) \to  f(v^K) \quad \textrm{in} \quad L^{\infty}((\eps,T);
   L^{s}(\Omega)) 
\end{displaymath}
and 
\begin{displaymath}
   g_{K}(u^K_n) \to  g_{K}(v^K) \quad \textrm{in} \quad L^{s}((\eps,T);
   L^{s}(\Gamma)) . 
\end{displaymath}

Hence,  letting  $n\to\infty$, we get 
\begin{displaymath}
  \frac{\mathrm{d}}{\mathrm{d}t}  \int_\Omega v^K \phi +
   \int_\Omega v^K (-\Delta \phi) +\int_\Omega  f(v^K) \phi 
=\int_\Gamma g_K(v^K)\phi 
\end{displaymath}
for $\epsilon \leq t \leq T$.

Then, notice that this is enough to guarantee that $v^K$ satisfies (see \cite{Ball1977})
\begin{displaymath}
  v^K(t) = S(t-\eps)v^K(\eps) + \int_\eps^t S(t-s)h(s)\,\textrm{d}s
\end{displaymath}
where $S(t)$ denotes the strongly continuous analytic semigroup generated by $\Delta$ in
$L^r(\Omega)$ with homogeneous Neumann boundary conditions, and
\begin{displaymath}
h =
-f_{\Omega}(v^K) + (g_K)_{\Gamma}(v^K)\in L^1([\eps,T];
L^r(\Omega)) +L^1([\eps,T]; L^r(\Gamma)).
\end{displaymath}

Finally, taking $\eps \to 0$ and using the continuity of the linear semigroup $S(t)$
and $v^K(\eps)\to v^K(0)=u_0$ in $L^{r}(\Omega)$ as $\eps\to 0$, we have that
$S(t-\epsilon) v^K(\eps) \to S(t)u_0$ in $L^{r}(\Omega)$ as $\epsilon \to 0$ 
and then,
\begin{displaymath}
  v^K(t) = S(t)u_0+ \int_0^t S(t-s)h(s)\,\textrm{d}s.
\end{displaymath} 
Hence, $v^K$ is a global  solution of  (\ref{eq:truncated_problem}) in
$L^r(\Omega)$ in the sense of the variations of constants formula.

\medskip 
\noindent {\bf Step 4.  Further regularity} 

From (\ref{eq:bound:vK:Linfty}), for any
$1<s\leq \infty$ and $\eps >0$,  we have that $v^{K}(\eps)\in
L^s(\Omega)$. Taking $s > r_{0}$ the problem (\ref{eq:truncated_problem}) is
subcritical in $L^s(\Omega)$. So, by part i) in  Theorem
\ref{th:attractor_subcritical},  the unique  solution of this problem starting at
$v^K(\eps)$, which is $v^K(t+\epsilon)$, is classical for $t>0$. Thus,
for $t> \eps$, $v^{K}(t)$ coincides with  
the solutions in Section \ref{sec:known-results-subcr}.

In particular,  $v^{K} \in C([\eps, T]\times \overline \Omega)$, for
every $\eps>0$, and  $v^K$ is a classical solution for $t>0$ and $v^K\in
C((0,T);C^2(\overline\Omega))$, for any $T>0$. 
\end{proof}

\begin{remark}
  If $f(0)=0=g(0)$ then $\Phi(t)=0$ and we can take $C(T,K)=0$ in
  (\ref{eq:bound:vK:Linfty}).
\end{remark}

Now we turn into uniqueness of solutions of (\ref{eq:truncated_problem}). 

\begin{theorem}
\label{th:uniqueness:truncate}
Let $1<r<r_{0}$ and $T>0$ fixed. Given $u_{0} \in L^{r}(\Omega)$,
there exists a unique function
\begin{displaymath}
v \in C([0,T];L^r(\Omega))\cap C([\eps, T]\times \overline \Omega))
, \quad    v(0) = u_{0} 
\end{displaymath}  
for any $\epsilon>0$,
 satisfying, for every $0<\eps\leq t \leq  T$, 
\begin{equation} \label{eq:VCF_truncated_problem}
    v(t) = S(t-\epsilon)v(\epsilon)+ \int_\epsilon^t S(t-s)\big( -f_{\Omega}(v(s)) +
    (g_K)_{\Gamma}(v(s)) \big) \,\mathrm{d}s, 
\end{equation}
where $S(t)$ denotes the semigroup generated by $\Delta$ with Neumann boundary
conditions in $L^r(\Omega)$.

In particular, the function $v^{K}(\cdot)$ constructed in Theorem
\ref{thr:truncated_problem_well_posed} is the unique 
function satisfying these conditions.

\end{theorem}
\begin{proof}
Let $v^K$ be the function  constructed in Theorem
\ref{thr:truncated_problem_well_posed} and let  $v$ be a function
such that $v\in C([0,T];L^r(\Omega))$, $v(0) = 
u_0$, and for any $\epsilon>0$, $v\in C([\epsilon,T];C(\overline\Omega))$ and
satisfies (\ref{eq:VCF_truncated_problem}).

So, fixed $\epsilon>0$,  $v(\epsilon)\in C(\overline{\Omega})$ and
then from (\ref{eq:VCF_truncated_problem}) and the results in Section
\ref{sec:known-results-subcr}, $v$   satisfies the equation
(\ref{eq:truncated_problem}) in $[\epsilon,T]$,  with initial data
$v(\epsilon)$, in a classical sense. In particular $v$ is smooth  for $t>0$.

Since both $v$ and $v^{K}$ satisfy (\ref{eq:truncated_problem}),
arguing as in (\ref{eq:Lipschitz_truncated_problem}) we get, for
$0<\eps\leq t \leq T$, 
\begin{displaymath}
  \frac{\mathrm{d}}{\mathrm{d}t} \|v^K(t)-v(t)\|^r_{L^r(\Omega)} +
  c_1 \|\nabla (|v^K(t)-v(t)|^{r/2})\|^2_{L^2(\Omega)}
  \leq C(K)   \|v^K(t)-v(t)\|^r_{L^r(\Omega)} .   
\end{displaymath}
Then, by Gronwall's Lemma, we have
\begin{displaymath}
  \|v^K(t)-v(t)\|_{L^r(\Omega)}^r 
  \leq C(K,T) \|v^K(\epsilon)-v(\epsilon)\|_{L^r(\Omega)}^r, 
  \quad  \quad 0<\eps\leq t \leq T.
\end{displaymath}
Since $v,v^K\in C([0,T];L^r(\Omega))$ and $v(0) = v^K(0) = u_0$, taking limits
as $\epsilon \to 0$ we have $v \equiv v^K$ on $[0,T]$.
\end{proof}

\begin{remark} \label{rem:continuity_at0_VCF}
  Since the function $v$ in  the theorem is continuous at $0$ in
  $L^r(\Omega)$,  we   can take $\eps \to  0$ in
  (\ref{eq:VCF_truncated_problem}), to obtain that  $v$ satisfies
  $v(0) = u_0$ and  
  \begin{displaymath}
     v(t) = S(t)u_0+ \int_0^t S(t-s)\big( -f_{\Omega}(v(s)) +
    (g_K)_{\Gamma}(v(s))\big) \,\textrm{d}s.
  \end{displaymath}

Conversely, if we know that the integral above makes sense, then a
simple algebraic manipulation implies that
(\ref{eq:VCF_truncated_problem}) holds true for any $\eps >0$. 

\end{remark}

\subsection{Uniform bounds in $K$ for the truncated problem}
\label{sec:unif-bounds-Lsigma}

In order to construct a solution of problem (\ref{eq:pbFG:monotonic}),
we are going to obtain uniform bounds in $K$ for the solutions of
(\ref{eq:truncated_problem}) constructed in Theorem
\ref{thr:truncated_problem_well_posed}, for 
positive times bounded away from zero. We fix $u_0\in L^r(\Omega)$ and
consider the family of solutions $v^K(\cdot;u_0)$ of
(\ref{eq:truncated_problem}) obtained in Theorem
\ref{thr:truncated_problem_well_posed}. Then we have the following.

\begin{proposition} \label{prop:bounds_vK_Lsigma}
  For any $1<r< r_{0}$ and $u_0\in L^r(\Omega)$, the solutions $v^K(\cdot;u_0)$
  of (\ref{eq:truncated_problem}) obtained in Theorem
  \ref{thr:truncated_problem_well_posed} satisfy
\begin{equation} \label{eq:estimate_vK_Lr}
\|v^K(t)\|_{L^r (\Omega)}  \leq \max\{
\|u_{0}\|_{L^r (\Omega)} ,   (\beta_{r}/\gamma_{r})^\frac{1}{r + p
  -1}\}, \qquad t\geq 0 , 
\end{equation}   
and for any $\sigma>r$, 
\begin{equation} 
\label{eq:estimate_vK_Lsigma} 
\|v^K(t)\|_{L^\sigma (\Omega)}
\leq \left({\beta_{\sigma} \over \gamma_{\sigma}}\right)^{1
     \over \sigma + p-1} + \left({\sigma 
    \over \gamma_{\sigma} (p-1) }\right)^{1 \over p-1} 
      t^{-\frac{1}{ p-1}} , \qquad t>0
\end{equation}
for some $\beta_{\sigma}, \gamma_{\sigma}>0$ depending on $\sigma\geq r$ but not in $K$ or $u_0$. 

Finally, 
\begin{displaymath}
   \int_0^T \int_\Gamma |v^K|^r \leq  C T
\end{displaymath}
and for every $0<\eps < T$ and  $\sigma >r$
\begin{equation} \label{eq:estimate_Lsigma_trace_vK}
   \int_\eps^T \int_\Gamma |v^K|^\sigma \leq
  C   T + \|v^K(\eps)\|_{L^{\sigma}(\Omega)}^{\sigma},
\end{equation} 
for some constants depending on $\sigma\geq r$ but not in $K$ or
$u_0$.

\end{proposition}
\begin{proof}{}
  
\medskip 
\noindent {\bf Step 1. Uniform bounds in $K$ in $C([\eps,\infty);L^\sigma(\Omega))$}

We show now that $v^K(t)$ are uniformly bounded, with respect to $K$, in
$L^\sigma(\Omega)$ for all $1< \sigma < \infty$. 

From the regularity of $v^{K}$ in Theorem
\ref{thr:truncated_problem_well_posed},  multiplying the equation in
(\ref{eq:truncated_problem})  by 
$|v^K|^{\sigma-2}v^K$ we have, see Proposition \ref{prop:norm:ineq:bounds},
\begin{eqnarray*}
  \frac{1}{\sigma}\frac{\mathrm{d}}{\mathrm{d}t} \|v^K(t)\|^\sigma_{L^\sigma(\Omega)} +
 c \|\nabla |v^K(t)|^{\sigma/2}\|^2_{L^2(\Omega)} 
 & + & \int_{\Omega} f(v^K(t)) |v^K(t)|^{\sigma-2}v^K(t) \\
 & = &  \int_\Gamma g_K(v^K(t))
  |v^K(t)|^{\sigma-2}v^K(t) . 
\end{eqnarray*}
Since $sg_K(s) \leq sg(s)$  for all $s\in\R$, see (\ref{eq:gKincreases_in_K}),
we have 
\begin{eqnarray*}
  \frac{1}{\sigma}\frac{\mathrm{d}}{\mathrm{d}t} \|v^K(t)\|^\sigma_{L^\sigma(\Omega)} +
  c \|\nabla |v^K(t)|^{\sigma/2}\|^2_{L^2(\Omega)} 
 & + & \int_{\Omega} f(v^K(t)) |v^K(t)|^{\sigma-2}v^K(t) \\
 & \leq &  \int_\Gamma g(v^K(t))
  |v^K(t)|^{\sigma-2}v^K(t) . 
\end{eqnarray*}
Finally, proceeding as in (\ref{eq:energy_NL}) we have
\begin{equation} \label{eq:mainEstimate:appr:Lsigma}
 {\frac{1}{\sigma}}  \frac{\mathrm{d}}{\mathrm{d}t} \|v^K(t)\|^\sigma_{L^\sigma(\Omega)} 
  +  {\frac{2(\sigma-1)}{\sigma^2}} \|\nabla(|v^K(t)|^{\sigma/2})\|_{L^2(\Omega)}^2+
  A \|v^K(t)\|^{\sigma+p-1}_{L^{\sigma+p-1}(\Omega)} \leq B
\end{equation}
with $A$  depending on the constants appearing on
(\ref{eq:mainHyp:fg}) but not on $\sigma, K$ or $u_{0}$ and $B>0$ depending on
$\sigma$ but not in $K$  or $u_{0}$.

In particular, denoting $y(t) =  \|v^K(t)\|^\sigma_{L^\sigma
  (\Omega)}$ and dropping the gradient term above,  we get that  $y(t)$
satisfies the following differential inequality
\begin{displaymath}
  \dot y (t)  + \gamma_{\sigma} y^{\frac{\sigma+p-1}{\sigma}}(t)\leq
  \beta_{\sigma} 
\end{displaymath}
for some $\beta_{\sigma}, \gamma_{\sigma}>0$ not depending on $K$ or
$u_{0}$. 

If $\sigma= r$, since $y(0) <\infty$, we have (\ref{eq:estimate_vK_Lr}). If
$\sigma > r$, let $z(t)$ be the solution of
\begin{displaymath}
  \dot z + \gamma_{\sigma} z^\frac{\sigma+p-1}{\sigma} = \beta_{\sigma}
\end{displaymath}
with $\lim_{t\to 0^{+}} z(t) = \infty$. 
Then from \cite{Temam1988}, Lemma 5.1, Chapter 3, page 167, we have 
\begin{displaymath}
  z(t) \leq \left({\beta_{\sigma} \over \gamma_{\sigma}}\right)^{\sigma \over \sigma + p-1} + {1
    \over \left(\gamma_{\sigma} {p-1 \over \sigma} t\right)^{\sigma \over p-1} } ,
  \quad t>0 . 
\end{displaymath}

Now, $0\leq y(t)\leq z(t)$ for all $0< t$ and then   we get, using
$a^\sigma + b^{\sigma} \leq (a+b)^\sigma$ for $a,b>0$,  
\begin{displaymath}
   \|v^K(t;u_0)\|_{L^\sigma (\Omega)} \leq \left({\beta_{\sigma} \over \gamma_{\sigma}}\right)^{1
     \over \sigma + p-1} + \left({\sigma 
    \over \gamma_{\sigma} (p-1) }\right)^{1 \over p-1} {1
    \over   t^{1 \over p-1} }  ,
  \quad t>0 .  
\end{displaymath}
Hence, we get
(\ref{eq:estimate_vK_Lsigma}). In particular, we obtain uniform
estimates in $L^{\sigma}(\Omega)$ for 
$t\geq \eps >0$.

\medskip 
\noindent {\bf Step 2.  Uniform bounds in $L^{\sigma}([\eps,T];L^\sigma(\Gamma))$}

Integrating (\ref{eq:mainEstimate:appr:Lsigma}) we get for any $0 <
\eps < T$, 
\begin{equation}
  \label{eq:mainEstimate:appr}
  \sup_{\eps\leq t \leq T}\|v^K(t)\|_{L^{\sigma}(\Omega)}^{\sigma} +
c_{1}  \int_\eps^T\left\|\, \nabla
    |v^K(s)|^{\sigma/2}\right\|_{L^2(\Omega)}^2\,\textrm{d}{s} +
  c_2 \int_{\eps}^{T} \|v^K(s)\|^{\sigma+p-1}_{L^{\sigma+p-1}(\Omega)}
  \,\textrm{d}{s} \leq
  c_{3} T + \|v^K(\eps)\|_{L^{\sigma}(\Omega)}^{\sigma}   
\end{equation}
for some constants depending on $\sigma$ but not on $K$ or $u_{0}$. 
Hence, (\ref{eq:estimate_Lr_trace}) and (\ref{eq:mainEstimate:appr}) give,
\begin{displaymath}
     \int_\eps^T \int_\Gamma |v^K|^\sigma \leq
   \int_{\eps}^{T} \|\nabla 
  (|v^K|^{\sigma/2})\|_{L^2(\Omega)}^2 +
  c \int_{\eps}^{T} \|v^K\|_{L^\sigma(\Omega)}^\sigma  \leq
  c_{3}
   T + \|v^K(\eps)\|_{L^{\sigma}(\Omega)}^{\sigma} . 
\end{displaymath}
and we get (\ref{eq:estimate_Lsigma_trace_vK}). 
\end{proof}

\begin{remark}
A careful analysis of the constants in (\ref{eq:estimate_vK_Lsigma}),
which can be traced back to   (\ref{eq:mainEstimate:appr:Lsigma}) and
(\ref{eq:energy_NL}), shows that we can not pass to the limit as
$\sigma \to \infty$ in (\ref{eq:estimate_vK_Lsigma}). Therefore with
the result above we are not able to obtain $L^{\infty}(\Omega)$
estimates, uniform in $K$. 

Having such estimates is very important as we now show. 

\end{remark}

We now show that a uniform $L^\infty(\Omega)$--bound for $\{v^K\}_K$ it is
enough to ensure that any limit of the family $\{v^K\}_K$ of solutions of
(\ref{eq:truncated_problem}) is a classical solution of
(\ref{eq:pbFG:monotonic}) for positive times for which the variations of
constant formula holds. We will use this result later to construct a solution of
(\ref{eq:pbFG:monotonic}). In fact we will show later that such uniform
$L^\infty(\Omega)$--bound  holds true. 

\begin{proposition}
\label{prop:regFromBounds}
Let $1<r< r_{0}$ and $u_0\in L^r(\Omega)$. Suppose that for any
$\epsilon >0$ and $T>0$ the solutions $v^K(\cdot;u_0)$
  of (\ref{eq:truncated_problem}) obtained in Theorem
  \ref{thr:truncated_problem_well_posed} satisfy 
\begin{equation}
\label{eq:bound:vK:enough}
  \|v^K\|_{L^\infty([\epsilon,T]\times \Omega)} \leq C(\epsilon,T)
\end{equation}
 with $C(\epsilon,T)$ not depending on $K$. Then, there
exists a subsequence of $\{v^K\}_{K}$, which we denote the same, such that
\begin{displaymath}
  v=   \lim_{K\to\infty}v^K  \quad \textrm{in} \quad C_\mathrm{loc}((0,\infty),
  H^{\alpha}_{\sigma}(\Omega)), \quad v(0)= u_{0},
\end{displaymath}
for any $r\leq \sigma$ and $\alpha<1+\frac{1}{\sigma}$. Moreover, for
any such subsequence, the limit function satisfies 
\begin{equation} \label{eq:estimate_v_Lr}
\|v(t)\|_{L^r (\Omega)}  \leq \max\{
\|u_{0}\|_{L^r (\Omega)} ,   \left({\beta_{r} \over \gamma_{r}}\right)^\frac{1}{r + p
  -1}\}, \qquad t\geq 0,
\end{equation}
and 
\begin{equation} \label{eq:estimate_v_Lsigma}
  \|v(t)\|_{L^\sigma (\Omega)} \leq \left({\beta_{\sigma} \over \gamma_{\sigma}}\right)^{1
     \over \sigma + p-1} + \left({\sigma 
    \over \gamma_{\sigma} (p-1) }\right)^{1 \over p-1} 
      t^{-\frac{1}{ p-1}} , \qquad   t>0,
\end{equation}
with $\beta_{\sigma}, \gamma_{\sigma}$ for $\sigma \geq r$ as in
(\ref{eq:estimate_vK_Lsigma}). Also,  for any $\eps>0$  
\begin{displaymath}
  \|v(t)\|_{H^{\alpha}_{\sigma} (\Omega)} \leq C(\eps)
 \quad \textrm{for all} \quad t\geq \eps ,
\end{displaymath}
with a bound independent of  $u_{0}$, for any $\sigma>1$
and $\alpha < 1 + {1\over \sigma}$. Furthermore $v$ satisfies the
variation of constants formula 
\begin{displaymath}
  v(t) = S(t-\eps)v(\eps) + \int_\eps^t S(t-s)\left(
  -f_{\Omega}(v(s)) + g_{\Gamma}(v(s))\right) \,\textrm{d}s , \quad t\geq \eps ,
\end{displaymath}
where $S(t)$ denotes the strongly continuous analytic semigroup generated by
$\Delta$ in $L^r(\Omega)$ with homogeneous Neumann boundary conditions.

\end{proposition}
\begin{proof}
  Observe that from the $L^{\infty}(\Omega)$ bounds, given $\eps>0$, there
  exists $K_0(\eps)$ such that for $K \geq K_{0}(\eps)$ we have
  $g_{K}(v^{K}(t)) =   g(v^{K}(t))$ for $t\geq \eps$, see (\ref{eq:increasing_interval}).  Then, for
  $t\geq \eps$, $v^{K}(t)$ is actually   a solution of
  (\ref{eq:pbFG:monotonic}) and this solutions lies, for $t\geq 
  \eps$, in a space in which (\ref{eq:pbFG:monotonic}) is subcritical.
 
  Then Theorem \ref{th:attractor_subcritical} with initial data $v^K(\eps)$,
  implies that, for $K \geq K_{0}(\eps)$,
  \begin{displaymath}
    \|v^K(t)\|_{H^{\alpha}_{\sigma} (\Omega)} \leq C(\eps)
    \quad \textrm{for all}\quad t\geq 2\eps 
  \end{displaymath}
  with a bound independent of $K$ and  $u_{0}$, for any 
  $\sigma>1$ and   $\alpha < 1 + {1\over \sigma}$.

  In particular, the bounds above for $t= 2 \eps$, imply that, by taking a
  subsequence (which we will denote the same) if necessary  we can assume that
  $v^{K}(2\eps)$ converges in $L^{\sigma}(\Omega)$ for some $\sigma
  >r_{0}$. Hence $v(2\eps) = \lim_{K}
  v^{K}(2\eps)$ in $L^{\sigma}(\Omega)$ with $\sigma >r_{0}$.

  Now, using again the bounds above, the fact that the functions $f$
  and $g$ are    Lipschitz on bounded sets of $\R$,
   and that  $v^{K}$, for $K\geq K_{0}(\eps)$,   satisfies the variations
  of constants formula  
\begin{displaymath}
  v(t) = S(t-2\eps)v(2\eps) + \int_{2\eps}^t S(t-s)\left(
  -f_{\Omega}(v(s)) + g_{\Gamma}(v(s))\right) \,\textrm{d}s, \quad
t\geq 2\eps ,
\end{displaymath}
  where $S(t)$ denotes the strongly continuous analytic semigroup generated by
  $\Delta$ in $L^r(\Omega)$ with homogeneous Neumann boundary
  conditions, we have that for any $T>0$ and
$\alpha<1+\frac{1}{\sigma}$, there exists $L(T)$ such that for $K_{1}, K_{2}
\geq K_{0}(\eps)$, and $t\in [2\eps, T]$,
  \begin{displaymath}
    t^{\alpha} \|v^{K_{1}}(t) - v^{K_{2}}(t)\|_{H_{\sigma}^{2\alpha}(\Omega)}
    \leq L(T) \|v^{K_{1}}(2\eps) - v^{K_{2}}(2\eps)\|_{L^{\sigma}(\Omega)}
  \end{displaymath}
  as $K_{1}, K_{2} \to \infty$ since the solutions of
  (\ref{eq:pbFG:monotonic}) 
  depend continuously on the initial data, see also (\ref{epsregsol2}).

  Therefore, since $ \|v^{K_{1}}(2\eps) -
  v^{K_{2}}(2\eps)\|_{L^{\sigma}(\Omega)} \to 0$, as $K_{1}, K_{2}\to
  \infty$, we obtain convergence of $v^{K}$ in $C([\eps,T];
    H^{\alpha}_{\sigma}(\Omega))$

Hence, taking $\eps_{n} \to 0$ and using a Cantor diagonal argument,
we conclude that there exists a subsequence, that we denote the 
  same $\{v^{K}\}_{K}$, such that
  \begin{displaymath}
    v=   \lim_{K\to\infty}v^K  \quad \textrm{in} \quad C_\mathrm{loc}((0,\infty);
    H^{\alpha}_{\sigma}(\Omega)), \quad v(0)= u_{0}
  \end{displaymath}
  for any $r\leq \sigma$ and $\alpha<1+\frac{1}{\sigma}$. 

  Moreover, from (\ref{eq:estimate_vK_Lr}), (\ref{eq:estimate_vK_Lsigma}) we get
  (\ref{eq:estimate_v_Lr}) and (\ref{eq:estimate_v_Lsigma}) . Also, and for any
  $\eps >0$
  \begin{displaymath}
    \|v(t)\|_{H^{\alpha}_{\sigma} (\Omega)} \leq C(\eps)
    \quad \textrm{for all}\quad t\geq \eps 
  \end{displaymath}
  with a bound independent of $u_{0}$, for any $\sigma>1$ and $\alpha < 1 +
  {1\over \sigma}$. Moreover, passing to the limit in the variation of
  constants formula satisfied by $v^{K}$, we get that $v$ also
  satisfies the variation of constants formula 
  \begin{displaymath}
    v(t) = S(t- 2\eps)v(2\eps) + \int_{2\eps}^t S(t-s)\big(
    -f_{\Omega}(v(\cdot)) + g_{\Gamma}(v(\cdot))\big) \,\textrm{d}s,
    \quad t\geq 2\eps
  \end{displaymath}
for any $\eps>0$,   where $S(t)$ denotes the strongly continuous
analytic semigroup generated by   $\Delta$ in $L^r(\Omega)$ with
homogeneous Neumann boundary conditions. 
\end{proof}

\begin{remark}

\noindent i) Note that as in Remark \ref{rem:continuity_at0_VCF}, if we knew that
the limit function $v$ is continuous at $t=0$ in $L^{r}(\Omega)$ then
we could take $\eps \to 0$ in the variation of constants formula.  

In fact, below we will prove a weaker form of continuity at $t=0$ in
Theorem \ref{thr:continuity_at_t=0}

\noindent ii) Note that there might be many limit functions $v$ in the
argument in Proposition \ref{prop:regFromBounds}.

However, for  nonnegative solutions we  will show below that the limit
function is unique;  see Proposition \ref{prop:positiveSolns}.

\end{remark}

\begin{remark} \label{rem:seria_masfacil}
Observe that one could be tempted to follow the following argument: From
the  bounds on $v^{K}$ in $L^{\sigma}(\Omega)$ in Proposition
\ref{prop:bounds_vK_Lsigma} for $\sigma >r_{0}$, which are uniform in
$t\geq \eps$, $K$ and $u_{0}$, we use Theorem
\ref{th:attractor_subcritical} in $L^{\sigma}(\Omega)$ where
(\ref{eq:pbFG:monotonic}) is subcritical and then we obtain the
uniform bounds (\ref{eq:bound:vK:enough}) in Proposition
\ref{prop:regFromBounds}. 

However the estimates in Theorem \ref{th:attractor_subcritical} are
not known to be uniform with respect to the nonlinear terms, which
would be changing with $K$ in the argument above. 

Hence, to obtain the uniform bounds (\ref{eq:bound:vK:enough}) in
Proposition \ref{prop:regFromBounds} we use some comparison argument
below.

\end{remark}

\subsection{Positive solutions}
\label{sec:positive-solutions}

We show now that in the case of dealing with positive solutions, we
have the bounds (\ref{eq:bound:vK:enough}). Moreover in this case we can pass to
the limit in $K$ (and not only in a subsequence)  and obtain a unique solution of problem
(\ref{eq:pbFG:monotonic}) for which the conclusions of Proposition
\ref{prop:regFromBounds} holds. 

The tools we use here  will also be used to obtain  uniform bounds for
 general sign changing  solutions of the problem as we will show in
 Section \ref{sec:sign-chang-solut}.

Assume that $f(0) \leq 0$ and $g(0) \geq 0$ so that  if $0\leq u_{0} \in
L^{s}(\Omega)$ with $s > r_{0}$ then the solutions of (\ref{eq:pbFG:monotonic})
and (\ref{eq:truncated_problem}) are nonnegative for all times, see
Appendix A in \cite{arrieta00:_attrac_unif_bounds}. Recall that when
$s > r_{0}$ both problems (\ref{eq:pbFG:monotonic}) and
(\ref{eq:truncated_problem}) are subcritical in $L^{s}(\Omega)$.

The following result shows in particular that the bounds
(\ref{eq:bound:vK:enough}) in Proposition \ref{prop:regFromBounds}
hold true. 

\begin{proposition} \label{prop:bounds_vK_positive}
Assume  $f(0) \leq 0$ and $g(0) \geq 0$.  Then we have

\noindent i) For any $1<r< r_{0}$ and  $0\leq u_{0} \in L^{r}(\Omega)$ the solution
$v^{K}(\cdot;u_{0})$ of (\ref{eq:truncated_problem}) given in Theorem
\ref{thr:truncated_problem_well_posed} is nonnegative for all times.

\noindent ii) For any $\eps >0$,   $v^{K}(t;u_{0})$
is bounded  in $L^{\infty}(\Omega)$,  uniformly in $t\geq \eps$,  $K$ and 
$u_{0} \in L^{r}(\Omega)$.

\end{proposition}
\begin{proof}{}
To prove part i)  note that it is enough to take $u_{0}^{n} \geq 0$ in
(\ref{eq:pbFG:monotonic:appr}) and then $u_{n}^{K}(t) \geq 0$ for all
$t\geq 0$, which,
by the convergence in Theorem \ref{thr:truncated_problem_well_posed}, 
implies $v^{K}(t)\geq 0$ for all $t\geq 0$. 

Now, to prove part ii), from (\ref{eq:estimate_vK_Lsigma}) we have that for any $\eps
>0$, $v^{K}(\eps)$ belongs to a bounded set in $L^{\sigma}(\Omega)$
which is independent of $K$ and  $u_{0}$. Taking $\sigma
>r_{0}$ and using that $g_{K}(s) \leq g(s)$ for $s \geq 0$, see
(\ref{eq:gKincreases_in_K}), we have that 
\begin{displaymath}
  0\leq v^{K}(t+\eps) \leq u(t; v^{K}(\eps)) , \quad t\geq 0
\end{displaymath}
where $u(t; v^{K}(\eps))$ denotes the solution of
(\ref{eq:pbFG:monotonic}) with initial data $v^{K}(\eps) \in
L^{\sigma}(\Omega)$, $\sigma >r_{0}$. 

Therefore the dissipativity results in Section
\ref{sec:known-results-subcr}, see Theorem
\ref{th:attractor_subcritical},  imply that  $u(t;
v^{K}(\eps))$  is bounded in $L^{\infty}(\Omega)$  uniformly in $t\geq
\eps$,  $K$ and $u_{0}$, 
and so is $v^{K}(t)$ for $t\geq
2\eps$. 

Since  $v^{K}(t)$ is smooth up to the boundary of $\Omega$ for $t>0$,
then it is also  bounded in $L^{\infty}(\Gamma)$  uniformly   for $t\geq
2\eps$, $K$ and  $u_{0}$.  
\end{proof}

\begin{remark} \label{rem:negativesolutions}

If  $f(0)\geq 0$ and $g(0) \leq 0$, a  similar argument gives the
bounds on  non-negative solutions. In fact  now   $g_{K}(s) \geq g(s)$ for $s
\leq 0$, and  we have that  
\begin{displaymath}
0\geq v^{K}(t+\eps) \geq u(t; v^{K}(\eps)) , \quad t\geq 0
\end{displaymath}
where $u(t; v^{K}(\eps))$ denotes the solution of
(\ref{eq:pbFG:monotonic}) with initial data $v^{K}(\eps)$. 
\end{remark}

Now using Proposition \ref{prop:regFromBounds} and the fact that the
solutions $v^{K}$ are nonnegative, we can pass to the limit as $K\to
\infty$. Note that below the full family $v^{K}$ converges and not only
a subsequence.

\begin{proposition}
 \label{prop:positiveSolns}
 Assume  $f(0)
\leq 0$, $g(0) \geq 0$, $1<r< r_{0}$ and  $0\leq u_{0} \in L^{r}(\Omega)$.
Then for the solutions
$v^{K}(\cdot;u_{0})$ of (\ref{eq:truncated_problem}) given in Theorem
\ref{thr:truncated_problem_well_posed}, the limit 
\begin{displaymath}
   v=   \lim_{K\to\infty}v^K
\end{displaymath}
exists where $v\geq 0$ is as in Proposition \ref{prop:regFromBounds}. 
\end{proposition}
\begin{proof}{} 
From Proposition \ref{prop:bounds_vK_positive} the functions $v^{K}$ are
nonnegative and from (\ref{eq:gKincreases_in_K}) $g_{K}(s)$ is
increasing in $K$ for $s\geq0$. Then it is not difficult to see that
for fixed $n\in \N$ the solutions in (\ref{eq:pbFG:monotonic:appr})
are nonnegative and increasing in $K$. Therefore the functions
$v^{K}$ in Theorem \ref{thr:truncated_problem_well_posed} are
increasing in $K$.

Using this and part ii) in Proposition \ref{prop:positiveSolns} we
have that  the limit  
\begin{displaymath}
0\leq v(t,x) = \lim_{K\to\infty}v^K(t,x)  \quad \mbox{exists pointwise
a.e. in $(0,T) \times \overline{\Omega}$},  
\end{displaymath}
and then $v$ coincides with the function $v$ in  Proposition
\ref{prop:regFromBounds} and the full family $v^{K}$ converges and not
only a subsequence. 
\end{proof}

\begin{remark} \label{rem:limit_negativesolutions}
If  $f(0) \geq 0$ and $g(0) \leq 0$, a  similar argument allows us to
pass to the limit on non-negative solutions (now  $v^K$ is decreasing
as $K \to \infty$).   
\end{remark}

\subsection{$L^\infty$ bounds for sign changing solutions}
\label{sec:sign-chang-solut}

Now we use the arguments in the former subsection to obtain the
uniform estimates (\ref{eq:bound:vK:enough}) for general sign changing
solutions $v^{K}(\cdot;u_{0})$ of (\ref{eq:truncated_problem}) given in Theorem
\ref{thr:truncated_problem_well_posed}.

\begin{proposition}
  \label{prop:unifBound:genSols}

Let $1<r< r_{0}$ and $u_0\in L^r(\Omega)$. Then for any $\eps >0$,
$v^{K}(t;u_{0})$ is bounded  in $L^{\infty}(\Omega)$,  uniformly in
$t\geq \eps$,  $K$ and   $u_{0} \in L^{r}(\Omega)$.  

Therefore  the uniform bounds (\ref{eq:bound:vK:enough}) hold and the
conclusions of Proposition 
\ref{prop:regFromBounds} apply.

\end{proposition}
\begin{proof}
From the assumptions on $f$ and $g$ it is clear that we can construct
$C^{1}(\R)$ functions $f^{-}$, $g^{+}$  satisfying
(\ref{eq:mainHyp:fg}),   $g^+(0)\geq 0$  and
  $g_K(s)\leq  g^+(s)$ for all $s\in\R$ and $K>0$ and  $f^-(0)\leq 0$
  and   $f(s)\geq  f^-(s)$ for all $s\in\R$. 

Now, from (\ref{eq:estimate_vK_Lsigma}) we have that for any $\eps
>0$, $v^{K}(\eps)$ belongs to a bounded set in $L^{\sigma}(\Omega)$
which is independent of $K$ and  $u_{0}$. Then we take  $\sigma
>r_{0}$ and   let  $0 \leq U(t; v^{K}(\eps))$ be the solution of the following problem
  \begin{displaymath}
    \left\{\begin{array}{rclcl}
        U_t  - \Delta U + f^-(U) & =
        & 0 & \mathrm{in} & \Omega\\
        \displaystyle\frac{\partial U}{\partial \vec{n}}&=&  g^+(U) &
        \mathrm{on} & \Gamma\\
        U(0) &=& |v^{K}(\eps)| \in L^{\sigma}(\Omega) . 
      \end{array}\right.
  \end{displaymath}
Then the comparison principle for subcritical problems, see Appendix A in
\cite{arrieta00:_attrac_unif_bounds}, implies that 
\begin{displaymath}
v^{K}(t+\eps) \leq U(t; v^{K}(\eps))  , \quad t\geq 0 . 
\end{displaymath}

Analogously, we can construct
$C^{1}(\R)$ functions $f^{+}$, $g^{-}$  satisfying
(\ref{eq:mainHyp:fg}),   $g^-(0)\leq 0$  and
  $g^-(s) \leq g_K(s)$ for all $s\in\R$ and $K>0$ and  $f^+(0)\geq 0$
  and   $f^{+}(s)\leq  f(s)$ for all $s\in\R$.  

Then let $W(t;-|v^{K}(\eps)|)\leq 0$ be the solution of the problem
\begin{displaymath}
    \left\{\begin{array}{rclcl}
        U_t  - \Delta U + f^+(U) & =
        & 0 & \mathrm{in} & \Omega\\
        \displaystyle\frac{\partial U}{\partial \vec{n}}&=&  g^-(U) &
        \mathrm{on} & \Gamma\\
        U(0) &=& -| v^{K}(\eps)| \in L^{\sigma}(\Omega) 
      \end{array}\right.
\end{displaymath}
 see Remark \ref{rem:negativesolutions}. 

Again, the comparison principle for subcritical problems implies
\begin{displaymath}
  W(t;-|v^{K}(\eps)|) \leq v^K(t+ \eps) \leq
  U(t;|v^{K}(\eps)|), \quad t\geq 0. 
\end{displaymath}

Therefore the dissipativity results in Section
\ref{sec:known-results-subcr}, see Theorem
\ref{th:attractor_subcritical},  imply that  $W(t;-|v^{K}(\eps)|)$
and $U(t;|v^{K}(\eps)|)$ are  bounded in $L^{\infty}(\Omega)$
uniformly in $t\geq 
\eps$,  $K$ and $u_{0}$, 
and so is $v^{K}(t)$ for $t\geq 2\eps$. 

Since  $v^{K}(t)$ is smooth up to the boundary of $\Omega$ for $t>0$,
then it is also  bounded in $L^{\infty}(\Gamma)$  uniformly   for $t\geq
2\eps$, $K$ and  $u_{0}$.  
\end{proof}

Observe that the results in Propositions \ref{prop:regFromBounds} and
\ref{prop:unifBound:genSols} allows us to define solutions of
(\ref{eq:pbFG:monotonic}) as follows. 

\begin{definition} \label{def:solution_supercritical}
  
For $1<r<r_0$ and $u_{0} \in L^{r}(\Omega)$ a solution of
(\ref{eq:pbFG:monotonic}) that we denote $u(t;u_{0})$ is any of 
the limit functions $v$ in Proposition
\ref{prop:regFromBounds}. 

\end{definition}

Note in case of nonnegative solutions, from Proposition
\ref{prop:positiveSolns},  such solution is unique and the same holds
for negative solutions, see Remark \ref{rem:limit_negativesolutions}.

\subsection{Continuity at $t=0$} 

In this section we prove that any solution of
(\ref{eq:pbFG:monotonic}) as in Definition
\ref{def:solution_supercritical} attains its initial data $u_{0}\in
L^{r}(\Omega)$,  in compact subsets of $\Omega$ but in a slightly weaker norm. In fact we have

\begin{theorem} \label{thr:continuity_at_t=0}
 
For $1<r<r_0$ and $u_{0} \in L^{r}(\Omega)$  any solution
$u(t;u_{0})$,  of
(\ref{eq:pbFG:monotonic}) as in Definition
\ref{def:solution_supercritical} satisfies for any $1\leq \alpha <r$ 
\begin{displaymath}
  u(t;u_{0}) \to u_{0} \quad \mbox{as $t\to 0$, in
    $L^{\alpha}_{loc}(\Omega)$} . 
\end{displaymath}

\end{theorem}
\begin{proof}
  Let $\Omega_{1}$ be a smooth open subset such that  $\overline{\Omega_{1}}
\subset \Omega$ and let $\varphi \in \mathcal{D}(\Omega)$ such that
$0\leq \varphi\leq 1$, $\varphi=1$ in $\Omega_{1}$ and $\Omega_{0}
=supp(\varphi)$ satisfies $\overline{\Omega_{1}}
\subset \Omega_{0}$ and $\overline{\Omega_{0}}
\subset \Omega$. 

Denote then $z^{K} = v^{K} \varphi$ which satisfies  
\begin{equation} \label{eq:truncated_equation_vK}
    \left\{\begin{array}{rclcl}
        z_t  - \Delta z  & = & -f(v^{K})\varphi  -2 \nabla v^{K}
        \nabla \varphi - v^{K} \Delta \varphi
         & \mathrm{in} & \Omega_{0}\\
        z&=& 0 &         \mathrm{on} & \partial\Omega_{0} \\
        z(0) &=& u_{0}  \varphi  \in L^{r}(\Omega_{0}) . 
      \end{array}\right.
\end{equation}

Now denote $h^{K}(t) = -f(v^{K})\varphi  -2 \nabla v^{K} \nabla \varphi -
v^{K} \Delta \varphi = h_{1}^{K}(t) + h_{2}^{K}(t) + h_{3}^{K}(t)$ and note that
from (\ref{eq:estimate_vK_Lr})  we have, for any $T>0$,  $h_{3}^{K} \in L^{\infty}([0,T],
L^{r}(\Omega_{0}))$, while $h_{2}^{K} \in L^{\infty}([0,T],
H^{-1,r}(\Omega_{0}))$ with norms bounded independent of $K$. 

On the other hand, 
observe that for any $\alpha \geq 1$
such that $\alpha p > r$ if we take $\sigma > \alpha p$ then we can
write
\begin{displaymath}
  {1 \over \alpha p} = {\theta \over \sigma} + {1-\theta \over r},
  \quad \theta \in (0,1) . 
\end{displaymath}
Then from (\ref{eq:estimate_vK_Lr}) and
(\ref{eq:estimate_vK_Lsigma}) and by
interpolation we get, using (\ref{eq:mainHyp:fg}),  
\begin{displaymath}
  \|f(v^{K})\|_{L^{\alpha}(\Omega)} \leq C \big(1 +
  \|v^{K}\|_{L^{\alpha p}(\Omega)}^{p}\big) \leq  C \big(1 +
  \|v^{K}\|_{L^{\sigma}(\Omega)}^{p\theta}
  \|v^{K}\|_{L^{r}(\Omega)}^{p(1-\theta)}\big) . 
\end{displaymath}
Hence for some $C$ independent of $K$ and for $0<t <T$,  
\begin{displaymath}
   \|f(v^{K})\|_{L^{\alpha}(\Omega)} \leq C \left(1 + {1 \over
     t^{p\theta \over p-1}} \right) = C \left(1 + {1 \over
     t^{p' \theta}} \right) . 
\end{displaymath}
Therefore, if $\theta p' <1$ we have  
\begin{displaymath}
h_{1}^{K} \in L^{1}([0,T],L^{\alpha}(\Omega_{0}))  
\end{displaymath}
with norms bounded independent of $K$.

Now we show that we can chose  $\alpha$ and $\sigma$ as  above. In
fact,  we have ${1 \over \alpha  p} = \theta 
  \big( {1 \over \sigma} - {1 \over r}\big) + {1 \over r}$ 
which,  using that $\sigma > r$, we write as 
\begin{displaymath}
 {p' \over r} - {p' \over \alpha  p} = \theta p'
  \left( {1 \over r} - {1 \over \sigma}\right)  .
\end{displaymath}
Hence the condition $\theta p' <1$ can be met provided 
\begin{displaymath}
0 < {1 \over \sigma} <   {p' \over \alpha p}     - {p'-1 \over r} = {1 \over
  p-1} \left( {1 \over \alpha} - {1 \over r }\right) 
\end{displaymath}
since  ${p' \over  p} = p' -1 = {1 \over  p-1}$. Note that in
particular this implies that $\alpha < r$. 

On the other hand, using the conditions  $\alpha \geq 1$,  $\alpha p > r$
and  $\sigma > \alpha p$,  we are bound to chose $\alpha, \sigma$ such
that 
\begin{displaymath}
0< {1 \over \sigma} <  \min\left\{  {1 \over
  p-1} \left( {1 \over \alpha} - {1 \over r }\right), {1 \over \alpha p
}\right\}, \quad  \max\left\{{r \over p }, 1\right\}  < \alpha < r. 
\end{displaymath}
Hence for any  $\alpha<r$ there exists such $\sigma$.  

Now the variations of constants formula for
(\ref{eq:truncated_equation_vK}) gives 
\begin{displaymath}
  z^{K}(t) = S_{D}(t) u_{0}\varphi + \int_{0}^{t}S_{D}(t-s) h^{K}(s) \, ds 
\end{displaymath}
where $S_{D}(t)$ denotes the strongly continuous analytic semigroup generated by
$\Delta$  with homogeneous  Dirichlet boundary conditions.

Now we use that  $h^{K} = h_{1}^{K} + h_{2}^{K} + h_{3}^{K}$ and denote $z_{1}$, $z_{2}$ and
$z_{3}$ the corresponding three terms resulting in the integral term
above. Then results on linear equations, see e.g. Theorem 4 in
\cite{ARB11:_pertur_banac}, gives that $z_{2}, 
z_{3} \in C([0,T], L^{r}(\Omega_{0}))$ and $z_{1} \in C([0,T],
L^{\alpha}(\Omega_{0}))$  with with norms bounded 
independent of $K$, $z_{i}(0)=0$ and
\begin{displaymath}
\sup\{\|z_{1}(t)\|_{L^{\alpha}(\Omega_{0})}, \|z_{2}(t)\|_{L^{r}(\Omega_{0})},
 \|z_{3}(t)\|_{L^{r}(\Omega_{0})}  , \ t\in [0,T]\}\leq C(T)  
\end{displaymath}
with $C(T)$ independent of $K$ and $C(T)\to 0$ as $T\to 0$.

Finally as $ S_{D}(t)
u_{0}\varphi  \to u_{0}\varphi$ in $L^{r}(\Omega_{0})$ as $t\to 0$, we
get that for any $\eps>0$ there exist $\delta>0$, independent of $K$,  such that 
\begin{displaymath}
  \|z^{K}(t) - u_{0}\varphi \|_{L^{\alpha}(\Omega_{0})} \leq \eps, \quad
  t \in (0,\delta] . 
\end{displaymath}
In particular, restricting to $\Omega_{1}$, we get that  
any $\eps>0$ there exist $\delta>0$, independent of $K$,  such that 
\begin{displaymath}
  \|v^{K}(t) - u_{0} \|_{L^{\alpha}(\Omega_{1})} \leq \eps, \quad   t \in
  (0,\delta] , 
\end{displaymath}
which reflects the local equicontinuity at $t=0$ of the family $v^{K}$. 

Therefore, passing to the limit along any subsequence as in
Proposition \ref{prop:regFromBounds}, we get that any solution of
(\ref{eq:pbFG:monotonic}) as in Definition
\ref{def:solution_supercritical} satisfies 
\begin{displaymath}
  \|v(t) - u_{0} \|_{L^{\alpha}(\Omega_{1})} \leq \eps, \quad   t \in
  (0,\delta]   
\end{displaymath}
and the result is proved. 
\end{proof}

\section{Asymptotic behavior}
\label{sec:asymptotic-behaviour}

Due to the strong smoothing properties obtained above we can actually
show now that all solutions of (\ref{eq:pbFG:monotonic}) as in
Definition \ref{def:solution_supercritical} regularize into a space in which
(\ref{eq:pbFG:monotonic}) is subcritical and then the asymptotic
behaviour of solutions is described by the global attractor in Theorem
\ref{th:attractor_subcritical}.

In fact, given a bounded set $B$ of initial data in $1<r<r_0$, from
Propositions 
\ref{prop:unifBound:genSols} and \ref{prop:regFromBounds} we have that
all solutions as in Definition \ref{def:solution_supercritical} starting at $B$ are 
uniformly bounded in $L^\sigma(\Omega)$ at any time $t \geq \eps$, for any
$1<\sigma<\infty$. Then, for $\sigma > r_{0}$,  problem
(\ref{eq:pbFG:monotonic}) is subcritical and we can use Theorem
\ref{th:attractor_subcritical}.

The following result summarises the results concerning the asymptotic
behaviour of solutions of properties (\ref{eq:pbFG:monotonic}) for
initial data in $L^{r}(\Omega)$ with $1<r<r_{0}$. 

\begin{theorem}
  \label{thm:existExtremal:supcrit}

Under the assumptions (\ref{eq:mainHyp:fg}) and
(\ref{eq:nonlinearbalance_fg}), for $1<r<r_{0}$ we have:

\noindent i) There exists  a compact invariant set $\mathcal{A}\subset
C^\beta(\overline\Omega) 
\cap H^{\alpha}_{s}(\Omega)\subset L^{r}(\Omega)$, for any $0\leq
\beta < 1$,  $s\geq 1$ and $0\leq \alpha <1 + \frac{1}{s}$, attracting
the solutions as in Definition \ref{def:solution_supercritical}
starting in bounded sets of $L^r(\Omega)$ in the norm of
$C^\beta(\overline\Omega)$ or $H^{\alpha}_{s}(\Omega)$.  

In particular, there exists an  absorbing
set in $C^\beta(\overline\Omega) \cap H^{\alpha}_{s}(\Omega)$. Also
$\mathcal{A} = W^{u}(E) $, that is,  the unstable set of
the set of equilibria of (\ref{eq:pbFG:monotonic}), $E$,  which is
nonempty. Hence $\mathcal{A}$ is independent of $r$ and coincides with
the set in Theorem \ref{th:attractor_subcritical}.

\noindent ii)   There exist two extremal equilibria $\varphi_m \leq
\varphi_{\mathcal{M}}$ for   problem (\ref{eq:pbFG:monotonic})  such
that $\varphi_m, \varphi_{\mathcal{M}} \in \mathcal{A}$. Hence any
stationary solution $\varphi$ of (\ref{eq:pbFG:monotonic}) satisfies 
\begin{displaymath}
\varphi_m(x) \leq \varphi(x) \leq \varphi_{\mathcal{M}}(x) , \quad
x\in \Omega . 
\end{displaymath}

Furthermore, 
  \begin{displaymath}
    \label{eq:extremalAttracting:subcrit}
    \varphi_m(x) \leq \liminf_{t\to\infty} u(t,x;u_0) \leq
    \limsup_{t\to\infty} u(t,x;u_0) \leq \varphi_{\mathcal{M}}(x)
  \end{displaymath}
  uniformly in $x\in\Omega$ and for $u_0$ in bounded sets of
  $L^r(\Omega)$. 

The   maximal equilibrium is (order-)stable from above and the minimal
one  from below.  

\end{theorem}
\begin{proof}
Note that in Propositions \ref{prop:positiveSolns} and
\ref{prop:unifBound:genSols} the bounds on
the functions $v^{K}$ are independent of $t\geq \eps$,  $K$ and  
$u_{0} \in L^{r}(\Omega)$.  Hence the bounds on $v(t;u_{0})$, as in
Proposition \ref{prop:regFromBounds} are also independent of $t\geq
\eps$ and  $u_{0} \in L^{r}(\Omega)$. Then for $t\geq \eps$ all
solutions of (\ref{eq:pbFG:monotonic}) starting at a bounded set of
initial data $B \subset L^{r}(\Omega)$ enter a bounded set in
$L^{\sigma}(\Omega)$ for $\sigma >r_{0}$. Thus, part  i) follows from
Theorem \ref{th:attractor_subcritical}. 

Now for part ii) note that from i) we have an absorbing set in
$L^{\infty}(\Omega)$ for all solutions of
(\ref{eq:pbFG:monotonic}). In particular, for some $M>0$ the ordered interval of
functions in $\Omega$ such that $-M \leq h(x) \leq M$ for all $x \in
\Omega$ in an absorbing interval. Hence we can use Theorem 1.1 in
\cite{Rodr'iguez-Bernal2008} get the existence of two
  extremal equilibria $\varphi_m, \varphi_{\mathcal{M}} \in L^r(\Omega)$. The
  maximal equilibrium is (order-)stable from above and the minimal one
  from below. See also Corollary 3.11 in
  \cite{Rodr'iguez-Bernal2008}. 
\end{proof}

\begin{remark}
\mbox{}

\noindent i)  In particular note that for $r=r_{0}$ $\mathcal{A}$
attracts bounded sets of $L^{r_{0}}(\Omega)$ and therefore Theorem
\ref{thm:existExtremal:supcrit} improves the results in Theorem
\ref{th:attractor_subcritical} for this critical space. 

\noindent ii) Restricting only to nonnegative solutions, we have
uniqueness for (\ref{eq:pbFG:monotonic}) as in Proposition
\ref{prop:positiveSolns}. Using that solutions are smooth after
positive time, Theorem 4.5 in \cite{Rodr'iguez-Bernal2008} allows a
more detailed description of the asymptotic behavior of such solutions. 

\end{remark}

\end{document}